\documentclass[11pt]{article}
\usepackage{amsmath, amsthm, amssymb,latexsym, enumerate,kotex}
\usepackage{graphicx}
\usepackage{hyperref}
\usepackage{xcolor}
\usepackage{cite}
\usepackage{dsfont}
\usepackage{graphicx, caption, subcaption}
\usepackage{subcaption} 
\usepackage{caption}
\usepackage{colortbl}
\usepackage{float}
 \usepackage{ulem}

\numberwithin{equation}{section}
\usepackage{array}
\newcolumntype{L}[1]{>{\raggedright\let\newline\\\arraybackslash\hspace{0pt}}m{#1}}
\newcolumntype{C}[1]{>{\centering\let\newline\\\arraybackslash\hspace{0pt}}m{#1}}
\newcolumntype{R}[1]{>{\raggedleft\let\newline\\\arraybackslash\hspace{0pt}}m{#1}}

\newcommand{\modfour}[1]{\mathrm{(}#1 \mathrm{\ mod\  4)}}

\newtheorem{theorem}{Theorem}[section]
\newtheorem{lemma}[theorem]{Lemma}

\newtheorem{claim}[theorem]{Claim}

\newtheorem{proposition}[theorem]{Proposition}

\title{On $2$-connected graphs avoiding cycles of length $0$ modulo $4$}
\author{Hojin Chu, Boram Park, Homoon Ryu}

\author{Hojin Chu\thanks{Korea Institute for Advanced Study, Seoul 02455, Republic of Korea. Email: hojinchu@kias.re.kr}, \and Boram Park\thanks{Department of Mathematics, Ajou University, Suwon 16499, Republic of Korea. Email:borampark@ajou.ac.kr},
\and Homoon Ryu\thanks{Department of Mathematics, Ajou University, Suwon 16499, Republic of Korea. Email: ryuhomoon@ajou.ac.kr}
}

\date{\today}

\begin{document}

\maketitle 

\begin{abstract} 
For two integers $k$ and $\ell$, an $(\ell \text{ mod }k)$-cycle means a cycle of length $m$ such that $m\equiv \ell\pmod{k}$. In 1977, Bollob\'{a}s proved a conjecture of Burr and Erd\H{o}s by showing that if $\ell$ is even or $k$ is odd, then every $n$-vertex graph containing no $(\ell \text{ mod }k)$-cycles has at most a linear number of edges in terms of $n$. Since then, determining the exact extremal bounds for graphs without $(\ell \text{ mod }k)$-cycles has emerged as an interesting question in extremal graph theory, though the exact values are known only for a few integers $\ell$ and $k$. 
Recently, Gy\H{o}ri, Li, Salia,  Tompkins, Varga and Zhu proved that every $n$-vertex graph containing no $(0 \text{ mod }4)$-cycles has at most $\left\lfloor \frac{19}{12}(n -1) \right\rfloor$ edges, and they provided extremal examples that reach the bound, all of which are not $2$-connected. In this paper, we show that  a $2$-connected graph without $\modfour{0}$-cycles has at most $\left\lfloor \frac{3n-1}{2} \right\rfloor$ edges, and this bound is tight by presenting a method to construct infinitely many extremal examples. 
\end{abstract}

\section{Introduction}
All graphs in the paper are simple. We denote by $e(G)$ the number of edges in a graph $G$.  
For two integers $k$ and $\ell$, an $(\ell \text{ mod }k)$-cycle or $(\ell \text{ mod }k)$-path means a cycle or path of length $m$ such that $m\equiv \ell\pmod{k}$.

Burr and Erd\H{o}s~\cite{BurrErdos} conjectured that if an $n$-vertex graph does not contain an $(\ell \text{ mod }k)$-cycle, where $k\mathbb{Z} + \ell$ contains an even number, then it has at most a linear number of edges in terms of $n$. 
This conjecture was later confirmed by Bollob\'{a}s~\cite{BBollobas}, whose result naturally led to the following question: What is the smallest constant $c_{\ell,k}$, where $k\mathbb{Z} + \ell$ contains an even number, such that every $n$-vertex graph with $c_{\ell,k}n$ edges contains an ($\ell$ mod $k$)-cycle? This problem has received significant attention, and some improvements to the general bounds on $c_{\ell,k}$ have been made. In particular, Sudakov and Verstra\"{e}te~\cite{Sudakov2017} showed that for $3\le \ell <k$, the value of $c_{\ell,k}$ is proportional to the maximum average degree of a $k$-vertex graph without cycles of length $\ell$. This determines $c_{\ell,k}$ up to an absolute constant. Moreover, for even $\ell \ge 4$, determining $c_{\ell,k}$ is as hard as determining the Tur\'{a}n number for even cycles. As a consequence, the exact asymptotics of $c_{\ell,k}$ are known up to constants only for $ \ell \in \{4,6,10\}$ or $\ell=\Omega(\log k)$.

The exact value of $c_{\ell,k}$ is known only for a few specific values of $\ell$ and $k$. It is well known that $c_{0,2} = \frac{3}{2}$. Chen and Saito~\cite{CS94} proved that $c_{0,3} = 2$ and the extremal graphs are $K_{2,n-2}$. Bai, Li, Pan and Zhang~\cite{BAI2025} proved that $c_{1,3} = \frac{5}{3}$, and one vertex identifications of Petersen graphs are extremal graphs. Dean, Kaneko, Ota and Toft~\cite{DKOT}, as well as Saito~\cite{Saito92}, showed that $c_{2,3} = 3$, with extremal graphs being $K_{3,n-3}$.

In the case of $(\ell \bmod 4)$-cycles, there has been substantial recent progress, and several noteworthy results have been announced. In~\cite{GLMX2024}, Gao, Li, Ma and Xie proved that an $n$-vertex graph with at least $\frac{5(n-1)}{2}$ edges contains two consecutive even cycles unless $4 \mid (n-1)$ and every block is isomorphic to $K_5$. This result not only shows that $c_{2,4} = \frac{5}{2}$, but also resolves a special case of a conjecture by Verstra\"{e}te~\cite{V2016} on extremal graphs avoiding $k$ cycles of consecutive lengths. Most recently, the result was extended by Li, Pan and Shi~\cite{LPS2025}, considering Sudakov-Verstra\"{e}te’s conjecture in \cite{Sudakov2017}.

In~\cite{0mod4cycle2025}, Gy\H{o}ri, Li, Salia, Tompkins, Varga and Zhu determined that $c_{0,4} = \frac{19}{12}$, as stated below, and the paper discusses extremal graphs achieving the bound, the examples given in~\cite{0mod4cycle2025} are not $2$-connected.

\begin{theorem}[\cite{0mod4cycle2025}]\label{thm:original}
Let $G$ be an $n$-vertex graph. If $e(G) > \left\lfloor \frac{19}{12}(n -1) \right\rfloor$, then $G$ contains a $\modfour{0}$-cycle.
\end{theorem}

When the known extremal graphs have cut-vertices, it naturally motivates the study of the extremal problem within the class of $2$-connected graphs. This line of research  is considered an interesting and meaningful study. For example, in~\cite{2conn2004}, Fan, Lv and Wang considered the extremal problem about 2-connected graphs avoiding a cycle of length at least $c$ for some given integer $c$, since the extremal graphs avoiding a cycle of length at least $c$ on general graphs have cut-vertices. 

In this paper, we study the extremal problem of avoiding $(0 \text{ mod } 4)$-cycles in the class of $2$-connected graphs. Dean, Lesniak and Saito~\cite{DLS93} showed that every $3$-connected graph always contains a $(0 \text{ mod } 4)$-cycle. As it was noted above, the extremal examples given in~\cite{0mod4cycle2025} are not $2$-connected, which immediately leads to the following question: What is the maximum number of edges in a $2$-connected graph that avoids $(0 \text{ mod } 4)$-cycles? Our main result gives a complete answer to this question as follows. 

\begin{theorem}\label{thm:main}
Let $G$ be a $2$-connected $n$-vertex graph. If $e(G) > \left\lfloor \frac{3n - 1}{2} \right\rfloor$, then $G$ contains a $\modfour{0}$-cycle.
\end{theorem}

The bound in Theorem~\ref{thm:main} is tight in the sense that we can construct infinitely many $2$-connected $n$-vertex graphs with exactly $\left\lfloor \frac{3n - 1}{2} \right\rfloor$ edges that contain no $\modfour{0}$-cycles, when $n\ge 12$ (see Section~4). 

The paper is organized as follows. Section 2 collects some definitions, observations, and states some known propositions. Section 3 gives a proof of Theorem~\ref{thm:main}, and Section 4 describes the extremal graphs.

\section{Preliminaries}

\subsection{Basic definitions and notations}

For a path $P: v_0v_1 \cdots v_m$, the subpath $v_iv_{i+1}\cdots v_j$ of $P$ is denoted by $P[v_i,v_j]$. 
When $C: v_0v_1\cdots v_mv_0$ is a cycle, the path $v_i v_{i+1}\cdots v_{i+j}$, where the subscripts are taken modulo $m+1$, is denoted by $C[v_{i},v_{i+j}]$. 
For a walk $W:v_0v_1 \cdots v_m$, $\overleftarrow{W}$ means a walk $v_m\cdots v_1v_0$ obtained by reversing $W$. The length of a walk $W$ is the number of edges in $W$, denoted by $\ell(W)$.
For two walks $P$ and $Q$, if the terminal vertex of $P$ and the initial vertex of $Q$ are the same, then $P+Q$ denotes a walk along $P$ and $Q$. For two vertex sets $X$ and $Y$ of a graph, a path $P$ is called an {\it $(X,Y)$-path} if one end of $P$ is contained in $X$, the other is contained in $Y$, and no interior vertex is contained in $X\cup Y$. When $X$ or $Y$ is a singleton, we drop the set notation from an $(X,Y)$-path for convenience. For example, if $X=\{x\}$ and $Y=\{y\}$, we call it an $(x,y)$-path. For nonempty subsets $X$ and $Y$ of $V(G)$,
we say that a set $S$ of vertices {\it separates} $X$ and $Y$ if every $(X,Y)$-path contains a vertex of $S$, and that $S$ is an {\it $(X,Y)$-separating set}. The following is a well-known Menger's Theorem. 

\begin{theorem}[\!\!{\cite[Theorem~2.2]{OBW13}}] \label{thm:menger}
Let $G$ be a graph, $X$ and $Y$ be subsets of $V(G)$. For every positive integer $k$, there are $k$ pairwise vertex-disjoint $(X,Y)$-paths in $G$ if and only if every $(X,Y)$-separating set contains at least $k$ vertices. 
\end{theorem}

The following is a folklore.  

\begin{proposition}\label{prop:folklore}
Let $G$ be a $2$-connected graph, $X = \{x, y \}$ be a vertex cut of $G$, and  $S$ be 
a connected component of $G-X$. If every $(x,y)$-path in $G[V(S) \cup X]$ has the same parity,  then $G[V(S) \cup X]$ is bipartite. 
\end{proposition}

\begin{proof}
Let $H = G[V(S) \cup X]$.
By Theorem~\ref{thm:menger}, for each edge $e=uv$ in $H$, when $e\neq xy$, there are two vertex-disjoint $(X,\{u,v\})$-paths $P$ and $Q$ in $G$ since $G$ is $2$-connected. 
Since $X=\{x,y\}$ is a vertex cut of $G$, $P$ and $Q$ must be paths in $H$. Hence, for an edge $e$ in $H$, if $e\neq xy$, then there is an $(x,y)$-path in $H$ containing the edge $e$.

Suppose to the contrary that $H$ has an odd cycle $C$ though every $(x,y)$-path in $H$ has the same parity. Then $xy\not\in E(C)$ even if $xy$ is an edge of $G$.
By the previous paragraph, there is an $(x,y)$-path  $P$ in $H$ containing an edge of $C$. 
Let $w_1$ and $w_2$ be the first and the last vertices of $P$ that are also in $V(P)\cap V(C)$. Then consider two $(x,y)$-paths $Q_1: P[x,w_1]+C[w_1,w_2]+P[w_2,y]$ and $Q_2: P[x,w_1]+\overleftarrow{C}[w_1,w_2]+P[w_2,y]$.
Clearly $\ell(Q_1)$ and $\ell(Q_2)$ have different parities, which is a contradiction.
\end{proof}

\subsection{Gadgets and small  graphs} 
For two vertices $u$ and $v$ of a graph $G$ and a nonempty set $L$ of nonnegative integers, we say a pair $(u,v)$ is an {\it $L$-type} in $G$ if for every $(u,v)$-path $P$ of $G$, $\ell(P) \equiv t \pmod{4}$ for some $t \in L$. We define six graphs $F_3$,  $F_4$, $F_6$, $F_7$, $F_8$, and $F_9$ in Figure~\ref{fig:gadget} with two specified vertices $a$ and $b$. 
It is clear that $(a,b)$ is a $\{2\}$-type in $F_3$. One can also observe that $(a,b)$ is a $\{0,3\}$-type in $F_6$, a $\{0,1\}$-type in $F_7$, a $\{1,2\}$-type in $F_8$, and a $\{2,3\}$-type in both $F_4$ and $F_9$.  Those six graphs $F_i$'s will play a key role in our proofs.

\begin{figure}[!ht]
\centering
\includegraphics[page = 1,height=4.2cm]{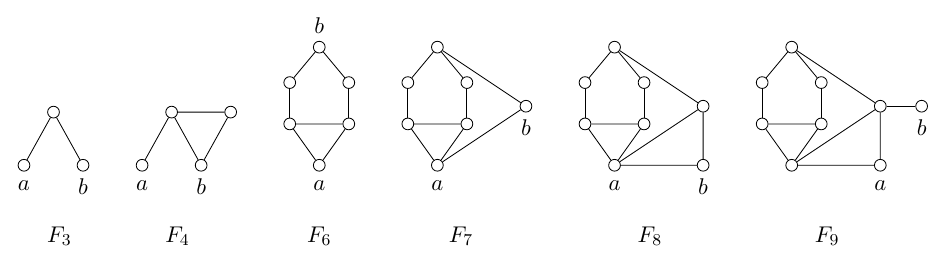}
\caption{Graphs $(F_i;a,b)$} \label{fig:gadget}
\end{figure}

We denote by $(G;x,y)$ a graph $G$ with two specified vertices $x$ and $y$.
Given two graphs $(G; x,y)$ and $(H; a,b)$, the {\it  parallel sum} of $(G;x,y)$ and $(H;a,b)$ is the graph obtained from $G\cup H$ by identifying $x$ and $a$, identifying $y$ and $b$, and then deleting a loop or multiple edges.
We denote it by $(G;x,y)\uplus (H;a,b)$. 
If $xy\in E(G)$ or $ab\in E(H)$, then the identified vertices are also adjacent in $(G;x,y)\uplus (H;a,b)$.  
If $x=y$ or $a=b$, then $(G;x,y)\uplus (H;a,b)$ is the graph obtained from $G\cup H$ by identifying all the vertices $x,y,a,b$ and deleting a loop or multiple edges.  
When $P$ is a path and the ends of $P$ are $a$ and $b$, we simply denote $(G;x,y)\uplus (P;a,b)$ by $(G;x,y)\uplus P$. 

For a vertex cut $\{x,y\}$ of a graph $G$, a connected component $S$ of $G-\{x,y\}$,
a \textit{reverse} of $G$ at $\{x,y\}$ with $S$ means a graph $H=(H_1;x,y)\uplus (H_2;y,x)$, where
$H_1=G[V(S)\cup \{x,y\}]$ and $H_2=G-V(S)$. 
We also say $H$ is obtained by \textit{reversing} $G$. 
Note that for a set $X$ of size two, $X$ is a vertex cut in $G$ if and only if $X$ is a vertex cut in $H$. If a graph $H$ can be obtained from $G$ by reversing a finite number of times, we say that $H$ is \textit{reserving-equivalent} to $G$. The two graphs in Figure~\ref{fig:reversing}  are reversing-equivalent to $F_9$ in Figure~\ref{fig:gadget}.
From the definition, 
when $G$ and $H$ are reversing-equivalent, it is clear that $e(G)=e(H)$, and $G$ has a cycle of length $\ell$ if and only if $H$ has a cycle of length $\ell$.
  
\begin{figure}[ht]
  \centering
\includegraphics[page=5,height=3cm]{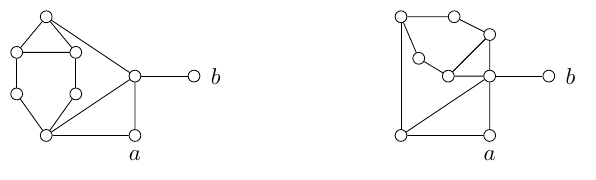}
\caption{Graphs that are reversing-equivalent to $F_9$ in Figure~\ref{fig:gadget}}\label{fig:reversing}
\end{figure}

We finish the section with observations on small graphs, whose proof is given in the appendix.
\begin{proposition}\label{prop:gadget}
Let $n\ge 3$ and let $(H;x,y)$ be an $n$-vertex graph without  $\modfour{0}$-cycles such that every edge in $H$ is contained in an $(x,y)$-path of $H$. 
If $(x,y)$ is an $L$-type in $H$, then \[
e(H) \le 
\begin{cases}
\frac{3n-4}{2} & \text{if } L=\{0,3\} \text{ and }  n \le 6;\\
\frac{3n-3}{2} & \text{if } L=\{0,1\} \text{ and }   n \le 7;\\
\frac{3n-2}{2} & \text{if } L=\{1,2\} \text{ and }   n \le 8;\\
\frac{3n-3}{2} & \text{if } L=\{2,3\} \text{ and }   n \le 9.
\end{cases}
\]
Moreover, if the equality holds, then $H$ is reversing-equivalent to $F_n$ in Figure~\ref{fig:gadget} for some $n\in\{ 6,7,8,9\}$.
\end{proposition}

\subsection{Odd cycles in a  graph without $\modfour{0}$-cycles}
  
This subsection gathers some results of previous work in \cite{DLS93,0mod4cycle2025}.

\begin{theorem}[\!\!\cite{DLS93}]\label{thm:3con}   Every $3$-connected graph has a $\modfour{\rm{0}}$-cycle.
\end{theorem}

A \textit{theta graph} $\Theta$ is a graph consisting of three internally vertex-disjoint paths $Q_1$, $Q_2$, $Q_3$ from a vertex $x$ to a vertex $y$, denoted by $\Theta(Q_1,Q_2,Q_3)$. We sometimes call it an $(x,y)$-theta graph.
In addition, if each path $Q_i$ has even length, then we call it an \textit{even} theta graph.

\begin{theorem}[\!\!\cite{0mod4cycle2025}]\label{thm:planar} The following holds.
\begin{itemize} 
\item[\rm(i)] {\rm{(Lemma 1)}} 
An even theta graph   contains a $\modfour{0}$-cycle.
\item[\rm(ii)] {\rm{(Lemma 2)}} Every non-planar graph  contains a $\modfour{0}$-cycle.
\item[\rm(iii)] {\rm{(Lemma 9)}}  If $G$ is a bipartite $n$-vertex graph  without $\modfour{0}$-cycles for some $n\ge 4$, then $e(G) \le 
\frac{3(n-2)}{2}$. 
\end{itemize}
\end{theorem}

\begin{lemma}[\!\!\cite{0mod4cycle2025}]\label{lem:previous}
Let $C_1$, $C_2$, and $C_3$ be odd cycles of a graph $G$ with $\ell(C_1) \equiv \ell(C_2) { \equiv \ell(C_3)}  \pmod 4$.
\begin{itemize}
\item[\rm(i)] {\rm{(Lemma 6(2))}} If $V(C_1) \cap V (C_2) = \{x\}$ and $P$ is an even  $(V(C_1),V(C_2))$-path with $x\not\in V (P)$, then $C_1 \cup C_2 \cup P$ contains a $\modfour{0}$-cycle.
\item[\rm(ii)] {\rm{(Lemma 6(3))}}
 If $C_1$, $C_2$ are vertex-disjoint  and $P$, $Q$, $R$ are vertex-disjoint $(V(C_1),V(C_2))$-paths, then $C_1 \cup C_2 \cup P \cup Q \cup R$ contains a $\modfour{0}$-cycle.
\item[\rm(iii)] {\rm{(Lemma 7)}}
Suppose that $C_1$, $C_2$, and $C_3$ pairwise intersect at a vertex $x$. 
Let $Q_i$ be a $(V(C_i),V(C_{i+1}))$-path in $G-V(C_{i+2})$ for each $i\in \{1,2,3\}$ (the subscripts are taken modulo $3$) such that $Q_1$, $Q_2$, and $Q_3$ are pairwise internally vertex-disjoint. 
Then $C_1 \cup C_2 \cup C_3 \cup Q_1 \cup Q_2 \cup Q_3$ contains a $\modfour{0}$-cycle. 
\end{itemize}
\end{lemma}

The following properties are easily derived from the above lemma and will be frequently used.

\begin{lemma}\label{lem:odd:cycles} 
Let $G$ be a $2$-connected graph without $\modfour{0}$-cycles, and $C_1$, $C_2$, and $C_3$ be three edge-disjoint odd cycles of $G$.
\begin{itemize}
\item[\rm(i)] There is no cycle $C$ such that $C \cap C_i$ induces a path in $C_i$ for each $i\in \{ 1, 2, 3\}$ and all vertices in $V(C_1)-V(C)$, $V(C_2)-V(C)$, and $V(C_3)-V(C)$ are distinct.

\item[\rm(ii)]  There is no theta graph $\Theta:=\Theta(Q_1,Q_2,Q_3)$ such that 
$\Theta \cap C_i$ induces a subpath of $Q_i$ for each $i \in\{ 1, 2, 3\}$ and all vertices in $V(C_1)-V(\Theta)$, $V(C_2)-V(\Theta)$, and $V(C_3)-V(\Theta)$ are distinct.    

\item[\rm(iii)] 
If $C_1$ and $C_2$ are triangles, then there are no three vertex-disjoint $(V(C_1),V(C_2))$-paths and 
$\ell(P)+\ell(Q)\equiv 3\pmod{4}$  for every two vertex-disjoint $(V(C_1),V(C_2))$-paths $P$ and $Q$.

\item[\rm(iv)] Suppose that $|V(C_i)\cap V(C_j)|=1$ for every distinct $i,j\in\{1,2,3\}$. Then $V(C_1) \cap V(C_2) \cap V(C_3)=\{v\}$ for some vertex $v$.
Moreover, if $\ell(C_1)\equiv\ell(C_2)\equiv \ell(C_3)\pmod{4}$, then there is no connected subgraph $H$ of $G-v$ such that 
$V(H)\cap V(C_i)\neq \emptyset$ and $E(H)\cap E(C_i)=\emptyset$ for each $i\in \{1,2,3\}$. 
\end{itemize}
\end{lemma}

\begin{figure}[!ht]
\centering
\includegraphics[page = 2, height=3.5cm]{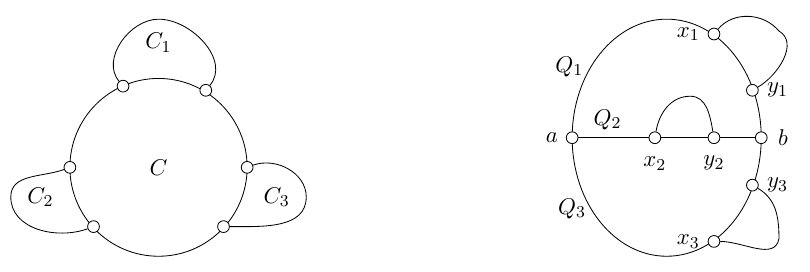}
\caption{Illustrations for Lemma~\ref{lem:odd:cycles}~(i)  and~(ii)}\label{fig:odd:cycles}
\end{figure}

\begin{proof}
(i) Suppose to the contrary that there is a cycle $C$ satisfying the condition.
See the first figure of Figure~\ref{fig:odd:cycles}. 
We may assume that $\ell(C_1)\equiv\ell(C_2)\pmod{4}$. 
Note that there are vertex-disjoint $(V(C_1),V(C_2))$-paths $P$ and $Q$ in $C\cup C_3$ such that $\ell(P)+\ell(Q)$ is even, since $C_3$ is an odd cycle. 
Since $\ell(C_1)+\ell(C_2) \equiv 2 \pmod{4}$, 
we can take subpaths $R_1$ and $R_2$ of $C_1$ and $C_2$, respectively, so that $\ell(R_1)+\ell(R_2)\equiv \ell(P)+\ell(Q)\pmod{4}$ and each $E(R_i)$ is either $E(C_i)\cap E(C)$ or $E(C_i)\setminus E(C)$.
By using $P$, $Q$, $R_1$, $R_2$, we can find a $\modfour{0}$-cycle, which is a contradiction.

\medskip

\noindent (ii) Suppose to the contrary that there is an $(a,b)$-theta graph $\Theta$ satisfying the condition.
Let $\Theta\cap C_i$ be an $(x_i,y_i)$-path and $Q'_i$ be an $(x_i,y_i)$-path whose edge set is $E(C_i) - E(\Theta)$, for each $i\in\{1,2,3\}$. See the second figure of Figure~\ref{fig:odd:cycles}. Then we can find an even $(a,b)$-theta graph by taking an even path among $Q_i$ or $Q_i[a,x_i]+Q'_i+Q_i[y_i,b]$, which is a contradiction to Theorem~\ref{thm:planar}~(i). 
\medskip

\noindent (iii)  Let $C_1$ and $C_2$ be triangles. If $V(C_1)$ and $V(C_2)$ are disjoint, then 
by Lemma~\ref{lem:previous}~(ii), there are no three vertex-disjoint $(V(C_1),V(C_2))$-paths.
In the following, if $V(C_1)$ and $V(C_2)$ are not disjoint, then we let $z$ denote the vertex in $V(C_1)\cap V(C_2)$.

We take any two vertex-disjoint $(V(C_1),V(C_2))$-paths $P$ and $Q$ by Theorem~\ref{thm:menger}.
Suppose that $V(C_1)$ and $V(C_2)$ are disjoint or one of $P$ and $Q$ is a trivial path $z$. By using paths of length one or two in $C_1$ and $C_2$, we find cycles of length $\ell(P)+\ell(Q)+2$, $\ell(P)+\ell(Q)+3$, and $\ell(P)+\ell(Q)+4$ in $G$. Therefore $\ell(P)+\ell(Q)\equiv 3 \pmod{4}$  since $G$ has no $\modfour{0}$-cycles. 

From the previous paragraphs, it is sufficient to show that when $C_1$ and $C_2$ share a vertex $z$, one of $P$ and $Q$ must be a trivial path. Suppose to the contrary that $C_1$ and $C_2$ share a vertex $z$ and both $P$ and $Q$ are nontrivial paths. Let $R$ be the trivial path with the vertex $z$. Then applying the previous paragraph to $P$ and $R$, we have $\ell(P)\equiv 3 \pmod{4}$. Similarly, $\ell(Q)\equiv 3\pmod {4}$.
Then $P$, $Q$, an edge of $C_1$, and an edge of $C_2$ form a $\modfour{0}$-cycle, which is a contradiction. 

\medskip

\noindent (iv) If  $|V(C_i)\cap V(C_j)|=1$ for every distinct $i,j\in\{1,2,3\}$ and $V(C_1) \cap V(C_2) \cap V(C_3) = \emptyset$, then one may easily find a {$\modfour{0}$-}cycle $C$ such that $V(C)\subset V(C_1)\cup V(C_2) \cup V(C_3)$, which contradicts (i) of this lemma.
Let $v$ be the vertex in $V(C_1)\cap V(C_2) \cap V(C_3)$, and let $B_i=C_i-v$ for each $i \in \{1, 2, 3\}$.  Suppose that $\ell(C_1)\equiv \ell(C_2)\equiv \ell(C_3)\pmod{4}$. Suppose to the contrary that there is a connected subgraph $H$ of $G-v$ such that $V(H)\cap V(C_i)\neq \emptyset$ and $E(H)\cap E(C_i)= \emptyset$ for each $i\in \{1,2,3\}$.
We take such $H$ as a smallest one. 
Then $H$ is a tree. Take a shortest $(V(C_i), V(C_{i+1}))$-path $R_{i,i+1}$ in $H$ for each $i\in \{1,2,3\}$, where subscripts are taken modulo $3$, and say $R_{i,i+1}$ is an $(x_i,y_i)$-path.
By Lemma~\ref{lem:previous}~(i), ${R_{i,j}}$ has odd length. 
Without loss of generality, we assume that $\ell(R_{1,2})\le \ell(R_{2,3})\le \ell(R_{3,1})$.

If $R_{1,2}$ contains a vertex of $B_3$ or $R_{2,3}$ contains a vertex of $B_1$, it contradicts the choice of $R_{3,1}$ and the assumption $\ell(R_{1,2})\le \ell(R_{2,3})\le \ell(R_{3,1})$.
Thus $R_{1,2}$ does not contain any vertex of $B_3$ and $R_{2,3}$ does not contain any vertex of $B_1$. Then $y_2$ is the only vertex in $V(B_3)\cap (V(R_{1,2})\cup V(R_{2,3}))$ and 
$x_1$ is the only vertex in $V(B_1)\cap (V(R_{1,2})\cup V(R_{2,3}))$.

Suppose that $y_1=x_2$. Then the walk $R_{1,2}+R_{2,3}$ contains at least one vertex of each $B_i$, and therefore, $E(H)=E(R_{1,2})\cup E(R_{2,3})$. Thus the path $R_{3,1}$ is a $(y_2,x_1)$-path together with observations in the previous paragraph.
Then {$R_{1,2}+R_{2,3}+R_{3,1}$} is a closed walk of an odd length. 
Since a closed walk of an odd length contains an odd cycle, it is a contradiction to the fact that $H$ is a tree. Thus $y_1\neq x_2$. 

Note that the existence of $z\in V(R_{1,2})\cap V(R_{2,3})$ implies that $R_{1,2}[x_1,z]+R_{2,3}[z,y_2]$ is a $(V(C_1),V(C_3))$-path, and therefore $E(R_{1,2})\cup E(R_{2,3}[z, y_2])$ induces a connected graph $H'$ in $G-v$ containing at least one vertex of each of 
$C_1$, $C_2$, and $C_3$. Since $H'$ is a subgraph of $H$ and $H'$ does not have an edge incident to $x_2$, $H'$ has less edges than $H$,  which is a contradiction. Then $V(R_{1,2})\cap V(R_{2,3}) = \emptyset$. 

If $R_{3,1}$ does not intersect with any of $R_{1,2}$ and $R_{2,3}$, {then} it is a contradiction to Lemma~\ref{lem:previous}~(iii).
Thus $R_{3,1}$ intersects with  $R_{1,2}$ or $R_{2,3}$. 
If $R_{3,1}$ intersects with  $R_{j,j+1}$ for some $j\in \{1,2\}$, then since $R_{3,1}\cup R_{j,j+1}$ does not have an edge incident to $x_2$ or $y_1$, it makes a 
proper connected subgraph $H'$ of $H$ such that
$V(H')\cap V(C_i)\neq \emptyset$ and {$E(H')\cap E(C_i)= \emptyset$} for each $i\in\{1,2,3\}$, which is a contradiction. 
\end{proof}

\begin{proposition}\label{clm:triperm}
Let $G$ be a $2$-connected graph without $\modfour{0}$-cycles, and $T_1$, $T_2$, $T_3$ be triangles in $G$. There is a permutation $\sigma$ on $\{1,2,3\}$ such that there are vertex-disjoint  $(V(T_{\sigma(1)}), V(T_{\sigma(2)}))$-paths $P$ and $Q$ satisfying $ab \in E(P)$ and $c \in V(Q)$ where $V(T_{\sigma(3)}) = \{a,b,c\}$. (See Figure~\ref{fig:prop2.9} for an illustration.)
\end{proposition}

\begin{figure}[h!]
\centering
\includegraphics[page=10]{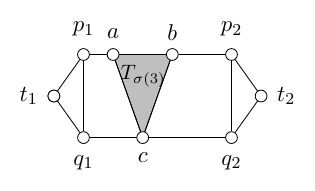}
\caption{An illustration for Proposition~\ref{clm:triperm}}\label{fig:prop2.9}
    \end{figure}

\begin{proof}  We note that since $G$ has no $C_4$, $|V(T_i) \cap V(T_j)| \le 1$ for every distinct $i,j\in\{1,2,3\}$.  First, suppose that  $|V(T_i)\cap V(T_j)|=1$ for every distinct $i,j\in\{1,2,3\}$. 
By Lemma~\ref{lem:odd:cycles}~(iv),  
there is a unique vertex $v$ in $V(T_1)\cap V(T_2)\cap V(T_3)$. 
Since $G-v$ is connected, we can take 
a smallest connected subgraph $H$ of $G-v$ that $V(H)\cap V(T_i)\neq \emptyset$ for every $i\in\{1,2,3\}$.
Since $H$ contains an edge of some $T_i$ by Lemma~\ref{lem:odd:cycles}~(iv), it gives a desired conclusion from the minimality of $H$.

In the following, we may assume that $V(T_1)\cap V(T_2)=\emptyset$. 
Since $G$ is $2$-connected, there are two vertex-disjoint $(V(T_1),V(T_2))$-paths $P$ and $Q$ by Theorem~\ref{thm:menger}. Let $P$ and $Q$ be a $(p_1,p_2)$-path and a $(q_1,q_2)$-path, respectively. We take such $P$ as a shortest one so that $|V(P)\cap V(T_3)|\ge |V(Q)\cap V(T_3)|$.
Let $\{t_i\} = V(T_i) - \{p_i, q_i\}$ for each $i \in \{1,2\}$. For simplicity, we also let $U = V(P) \cup V(Q) \cup V(T_1) \cup V(T_2)$.  

\medskip

\noindent{(Case 1)} Suppose that $|V(T_3) \cap U|\le 1$. By Theorem~\ref{thm:menger}, there are vertex-disjoint $(V(T_3), U)$-paths $R_1$ and $R_2$. Let $R_1$ and $R_2$ be an $(a,r_1)$-path and a $(b,r_2)$-path, say $R_1$ is a shortest one.
Note that $R_1$ might be a trivial path and $R_2$ cannot be a trivial path by our case assumption that $|V(T_3) \cap U|\le 1$.
If $\{r_1,r_2\}=\{t_1,t_2\}$, then $P$, $Q$, and $\overleftarrow{R_1}+ab+R_2$ are three vertex-disjoint $(V(T_1), V(T_2))$-paths, which is a contradiction by Lemma~\ref{lem:odd:cycles}~(iii). Without loss of generality, we may assume that $t_2\not\in \{r_1,r_2\}$. Then one of the following four cases must hold.
\begin{itemize}
\item[(1)-1] $r_1=t_1$ and $r_2\in \{p_1,q_1\}$ (The case where $r_1 \in \{p_1,q_1\}$ and $r_2=t_1$ is similar.);
\item[(1)-2] $r_1=t_1$ and $r_2\in (V(P)\cup V(Q))\setminus\{p_1,q_1\}$ (The case where $r_1 \in (V(P)\cup V(Q))\setminus\{p_1,q_1\}$ and $r_2=t_1$ is similar.);
\item[(1)-3] $\{r_1,r_2\}\subset V(P)$ (The case where $\{r_1,r_2\}\subset  V(Q)$ is similar.);
 \item[(1)-4] $|V(P)\cap\{r_1,r_2\}|=|V(Q)\cap\{r_1,r_2\}|=1$.
\end{itemize}
For the case (1)-1, we let $\sigma=(13)$ , which satisfies the desired condition of the proposition.
For the cases (1)-2, (1)-3, (1)-4, using $R_1$ and $R_2$, we reach a contradiction by finding $C$ or $\Theta$ described in Lemma~\ref{lem:odd:cycles}~(i) or (ii).
See Figure~\ref{fig:proof:PQ}. For the cases (1)-2 and (1)-3, the thick lines show a cycle $C$ described in Lemma~\ref{lem:odd:cycles}~(i). For the case (1)-4, the thick lines show an $(r_1,r_2)$-theta graph $\Theta$ described in Lemma~\ref{lem:odd:cycles}~(ii).

\begin{figure}[!ht]
\includegraphics[page=3,width=16.5cm]{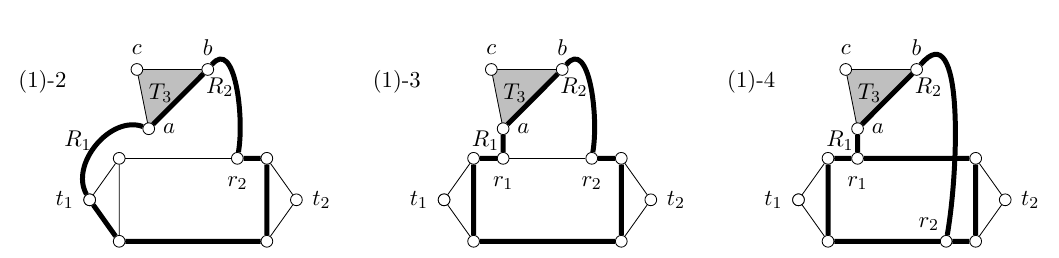}
\caption{Illustrations for (Case 1)}\label{fig:proof:PQ}
\end{figure}

\noindent{(Case 2)} Now suppose that $|V(T_3) \cap U|\ge 2$. 
If $V(T_3) \cap U$ contains $\{t_1,t_2\}$, then $P$, $Q$, and $t_1t_2$ are three vertex-disjoint $(V(T_1), V(T_2))$-paths, which is a contradiction by Lemma~\ref{lem:odd:cycles}~(iii).
Thus we may assume that $t_2\not\in V(T_3)$. Then $|V(T_3)\cap V(P)|\ge 1$ by the choice of $P$, and so let $a$ be the vertex in $V(T_3)\cap V(P)$ that is closest to $p_1$. 
By  the minimality of $\ell(P)$, $|V(T_3)\cap V(P)|\le 2$.
We suppose that it is not the case where $|V(T_{3}) \cap V(P)|= 2$ and $|V(T_{3}) \cap V(Q)|=1$, since it is the desired condition of the proposition.
Then one of the following four cases must hold.
\begin{itemize} 
\item[(2)-1]  $|V(T_{3}) \cap V(P)|= 2$, $t_1 \in V(T_3)$, $|V(T_{3}) \cap V(Q)|=0$;
\item[(2)-2]  $|V(T_{3}) \cap V(P)|= 2$, $t_1 \not\in V(T_3)$, $|V(T_{3}) \cap V(Q)|=0$;
\item[(2)-3]  $|V(T_{3}) \cap V(P)|=1$, $t_1\not\in V(T_3)$ (and therefore $|V(T_{3}) \cap V(Q)|=1$ by the case assumption);
\item[(2)-4]  $|V(T_{3}) \cap V(P)|=1$, $t_1\in V(T_3)$, $|V(T_{3}) \cap V(Q)|=0$;
\item[(2)-5]  $|V(T_{3}) \cap V(P)|=1$, $t_1\in V(T_3)$, $|V(T_{3}) \cap V(Q)|=1$.
\end{itemize}
For the cases (2)-1, (2)-2, (2)-3, and (2)-4, we reach a contradiction by finding $C$ or $\Theta$ described in Lemma~\ref{lem:odd:cycles}~(i) or (ii). See Figure~\ref{fig:proof:PQ:2}. For the cases (2)-1, (2)-2, and (2)-4, the thick lines show a cycle $C$ described in Lemma~\ref{lem:odd:cycles}~(i).
For the case (2)-3, the thick lines show a theta graph $\Theta$ described in Lemma~\ref{lem:odd:cycles}~(ii). 
For the case (2)-5, see the last figure of Figure~\ref{fig:proof:PQ:2}, there are  three  vertex-disjoint $(V(T_1), V(T_3))$-paths, which is a contradiction to Lemma~\ref{lem:odd:cycles}~(iii).
Note that in the cases (2)-1, (2)-4 and (2)-5, the vertices $p_1$ and $a$ are distinct, since two triangles share at most one vertex.

\begin{figure}[!ht]
\centering
\includegraphics[page=4,width=17cm]{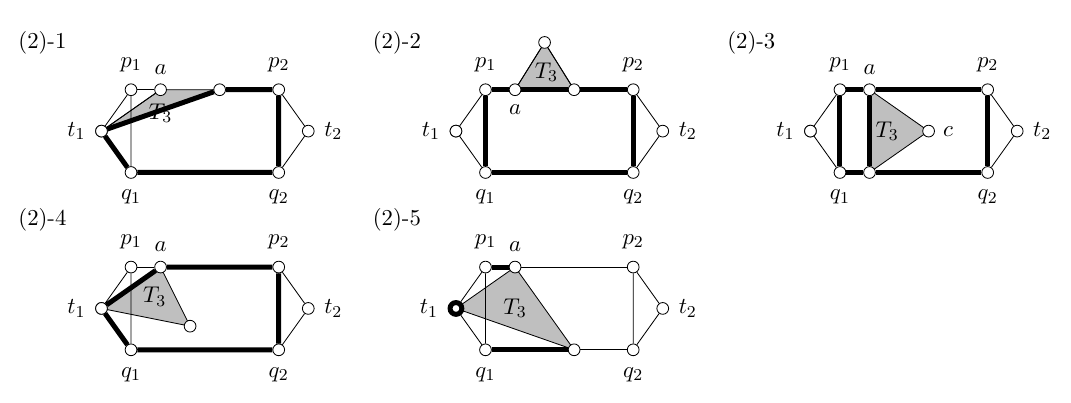}     
\caption{Illustrations for (Case 2)}\label{fig:proof:PQ:2}
\end{figure}
\end{proof}

\section{Proof of Theorem~\ref{thm:main}}
Let $G$ be a minimal counterexample to Theorem~\ref{thm:main}, and let $n=|V(G)|$. 
That is, $G$ is a $2$-connected graph that contains no $\modfour{0}$-cycle, $e(G) > \frac{3n-1}{2}$, and every $2$-connected graph $G'$ of order less than $n$ has at most $\frac{3|V(G')|-1}{2}$ edges. It is easy to observe that if $n\le 5$, then $G$ must be a cycle, and so we assume that $n\ge 6$ in the following.
Then by Theorem~\ref{thm:planar}~(ii), $G$ is planar, and assume that $G$ is with a fixed plane embedding.
The following lemma holds, where its proof is given later.

\begin{lemma}\label{lem:main}
    In the graph $G$, there are at most two $3$-faces and at most five $5$-faces.
\end{lemma}

Let $f(G)$ be the number of faces of $G$, and $f_i$ be the number of $i$-faces of $G$.
Then   $f_3 \le 2$ and $f_5 \le 5$ by Lemma~\ref{lem:main}.
Since $G$ has no $\modfour{0}$-cycles, $f_4=0$.
Thus, by Euler's formula,
\[ n+f(G) =2+e(G)=2+\frac{1}{2}\sum_{i\ge 3}if_i \ge 2+ 3f(G) -\frac{3}{2}f_3-f_4-\frac{1}{2}f_5 \ge 3f(G)-\frac{7}{2}  \]
and so  $f(G) \le \frac{2n+7}{4}$. 
Since $f(G)$ is an integer, $f(G) \le \frac{2n+6}{4}=\frac{n+3}{2}$.
Therefore,  $e(G) = n+ f(G)-2  \le \frac{3n-1}{2}$,
which is a contradiction.

To prove Lemma~\ref{lem:main}, we first examine useful structural properties of a vertex cut  in our minimal counterexample $G$ in Subsection 3.1, and then establish the proof of the lemma in Subsection 3.2.
 
\subsection{The structure of vertex cuts of $G$}

In this subsection, we aim to explore the structural properties of a vertex cut in the minimal counterexample $G$. 
By Theorem~\ref{thm:3con},  $G$ is not $3$-connected.  By Theorem~\ref{thm:planar}~(iii), $G$ is non-bipartite. The following lemma is the final goal of the subsection. We often recall the six graphs $F_i$'s in Figure~\ref{fig:gadget}.

\begin{lemma}\label{lem:no_F_6}
Let $X=\{x,y\}$ be a vertex cut of $G$. Then $G-X$ has exactly two connected components and there exists a connected component $S$ of $G-X$ such that $G[V(S) \cup X]$ is one of $K_3$, $F_3$, and $F_4$.
\end{lemma}

We start with a simple observation.

\begin{lemma}\label{lem:basic:alpha:4}
Let $X=\{x,y\}$ be a vertex cut of $G$, and let $(Z_1,Z_2)$  a bipartition of $V(G)-X$ such that there are no edges between $Z_1$ and $Z_2$. 
Suppose that $H_i=G[Z_i\cup X]$ and $e(H_i)\le \frac{3|V(H_i)|-1-\alpha_i}{2}$ for each $i\in\{1,2\}$. Then $\alpha_1+\alpha_2\le 4-2e(G[X])$. 
\end{lemma}
\begin{proof}
Suppose that $\alpha_1+\alpha_2\ge 5-2e(G[X])$. Note that $|V(H_1)|+|V(H_2)|=n+2$.
Then \[ e(G)\le e(H_1)+e(H_2)-e(G[X])\le \frac{3(|V(H_1)|+|V(H_2)|)-2-(\alpha_1+\alpha_2+2e(G[X]))}{2} \le \frac{3n-1}{2},\]
which contradicts the choice of $G$.
\end{proof}

When $S$ is a connected component of $G-X$ for some vertex cut $X=\{x,y\}$ of $G$, every edge in $G[V(S)\cup X]$ is contained in an $(x,y)$-path of $G[V(S)\cup X]$. Therefore,  we often apply Proposition~\ref{prop:gadget} to $G[V(S)\cup X]$ when $G[V(S)\cup X]$ has a small number of vertices. 
In addition, Table~\ref{table:basic:1201/2303} shows the summations of the integers modulo $4$, where each cell collects $\ell_1+\ell_2  \pmod{4}$ for every $\ell_1$ and $\ell_2$
from the sets indicated by its row and column, respectively.

\begin{table}[h!]
\centering
\begin{tabular}{c||c|c|c|c|}
+ & $\{0,1\}$ &$\{1,2\}$  &$\{2,3\}$ & $\{0,3\}$  \\ \hline \hline
   $\{0,1\}$  & $\{0,1,2\}$ &$\{1,2,3\}$&$\{0,2,3\}$&$\{0,1,3\}$\\
    $\{1,2\}$&  &$\{0,2,3\}$&$\{0,1,3\}$&$\{0,1,2\}$\\
   $\{2,3\}$ & & &$\{0,1,2\}$&$\{1,2,3\}$\\
  $\{0,3\}$&  &  &  &$\{0,2,3\}$ \\
\end{tabular}
\caption{Summation table, taken modulo 4.}\label{table:basic:1201/2303}
\end{table}  

\begin{lemma} \label{lem:threecomp} 
Let $X=\{x,y\}$ be a vertex cut of $G$.
Then $G-X$ has exactly two connected components.
\end{lemma}

\begin{proof}
Let $S_1,\ldots, S_r$ be the connected components of $G-X$, $G_i=G[V(S_i) \cup X]$ and $n_i=|V(G_i)|$ for each $i\in\{1,\ldots,r\}$. Suppose to the contrary that $r\ge 3$.
If there exists an $(x,y)$-path in $G_i$ of even length for each $i\in \{1,2,3\}$, then those three paths form an even theta graph, which is a contradiction to Lemma~\ref{thm:planar}~(i). 
Thus we may assume that every $(x,y)$-path in $G_1$ has odd length. 
By Proposition~\ref{prop:folklore}, $G_{1}$ is bipartite.
Clearly $n_1\ge 4$ and so by Theorem~\ref{thm:planar}~(iii),
$e(G_{1})\le \frac{3|V(G_1)|-6}{2}$. Note that $H_2=G-V(S_1)$ is a $2$-connected graph from our assumption that $r\ge 3$. By the minimality of $G$, $e(H_2)\le \frac{3|V(H_2)|-1}{2}$. We reach a contradiction to Lemma~\ref{lem:basic:alpha:4}. Hence, $r=2$.
\end{proof}

\begin{lemma} \label{lem:threecomp:2} 
Let $X=\{x,y\}$ be a vertex cut of $G$ such that $S_1$ and $S_2$ are the connected components of $G-X$. Let $G_i=G[V(S_i) \cup X]$, $n_i=|V(G_i)|$, and $L_i$ be the smallest subset of $\{0,1,2,3\}$ such that $(x,y)$ is an $L_i$-type in $G_i$ for each $i\in\{1,2\}$. Then each of the following holds.
\begin{itemize}
\item[\rm(i)]  If $L_1=\{0,1\}$ and $L_2=\{1,2\}$, then 
$n_2=3$. 
\item[\rm(ii)] For each $i\in\{1,2\}$, either $G_i$ is non-bipartite or $n_i=3$.
\end{itemize}
\end{lemma}

\begin{proof}
(i)  Suppose to the contrary that $L_1=\{0,1\}$, $L_2=\{1,2\}$, and $n_2 \ge 4$.
Let $H_1=(G_1;x,y) \uplus (K_3;a,b)$, where $a,b$ are two vertices of $K_3$, and 
let $H_2=G_2+xy$, which is the graph obtained by adding an edge $xy$ to $G_2$. 
Then both $H_1$ and $H_2$ are $2$-connected and have no $\modfour{0}$-cycles. 
Since $n_2 \ge 4$, it follows that $|V(H_1)|=n_1+1<n$.
Therefore, by the minimality of $G$, 
\[
e(G_1) = e(H_1)-3+e(G[X])\le \frac{3(n_1+1)-1}{2} - 3+e(G[X])=\frac{3n_1-4+2e(G[X])}{2}.\]
If $xy \notin E(G)$, then $e(G_1)\le \frac{3n_1-4}{2}$ and  $e(G_2) = e(H_2)-1 \le \frac{3n_2-1}{2}-1=\frac{3n_2-3}{2}$,
which contradicts to Lemma~\ref{lem:basic:alpha:4}. 
Therefore, $xy \in E(G)$, that is, $e(G[X])=1$. 
It also implies that  $e(G_1) \le \frac{3n_1-2}{2}$ and both $G_1$ and $G_2$ are $2$-connected. 
By the minimality of $G$, $e(G_2) \le \frac{3n_2-1}{2}$.

Suppose that $n_1 \le 7$ or $n_2\le 8$. 
Then either $e(G_1) \le \frac{3n_1-3}{2}$
or $e(G_2) \le \frac{3n_2-2}{2}$
by Proposition~\ref{prop:gadget}.
If $G_1$ is not reversing-equivalent to $F_7$ or $G_2$ is not reversing-equivalent to $F_8$, then by Proposition~\ref{prop:gadget} again, 
either $e(G_1) \le \frac{3n_1-4}{2}$
or $e(G_2) \le \frac{3n_2-3}{2}$, 
which contradicts Lemma~\ref{lem:basic:alpha:4}. 
If $G_1$ is reversing-equivalent to $F_7$ and $G_2$ is reversing-equivalent to $F_8$, then  $|V(G)|=13$ and $e(G)=19$, a contradiction. 

Suppose that $n_1 \ge 8$ and $n_2 \ge 9$. 
Let $H_1=(G_1;x,y) \uplus (F_8;a,b)$ and $H_2=(G_2;x,y) \uplus (F_7;a,b)$. 
Then for each $i\in \{1,2\}$, $|V(H_i)|< n$ and $H_i$ is a $2$-connected graph without $\modfour{0}$-cycles. 
By the minimality of $G$,
\[e(G_1)=e(H_1)-e(F_8)+1 \le \frac{3(n_1+6)-1}{2}-11+1=\frac{3n_1-3}{2}\]
\[e(G_2)=e(H_2)-e(F_7)+1 \le \frac{3(n_2+5)-1}{2}-9+1=\frac{3n_2-2}{2},\]
which contradicts Lemma~\ref{lem:basic:alpha:4}. 
Therefore, $n_2=3$. 

\medskip

\noindent (ii)  Suppose that to the contrary that there is a vertex cut $X$ such that $G_1$ is bipartite and $n_1\ge 4$. We choose a vertex cut $X$ of $G$ so that $n_1$ is maximized. 
Then $e(G_1)\le \frac{3n_1-6}{2}$ by Theorem~\ref{thm:planar}~(iii). If $G_2$ is $2$-connected, then $e(G_2)\le \frac{3n_2-1}{2}$ by the minimality of $G$, which is a contradiction to Lemma~\ref{lem:basic:alpha:4}. 
Then $G_2$ has a cut-vertex $z$ and it also holds that $xy\not\in E(G)$. Note that $z\not\in \{x,y\}$ and each of $\{x, z\}$ and $\{y,z\}$ is a vertex cut of $G$. 
Since $z$ is not a cut-vertex of $G$, $G_2-z$ has exactly two connected components $S'_1$ and $S'_2$ such that $x\in V(H_1)$ and $y\in V(H_2)$ where $H_i=G[V(S'_i)\cup \{z\}]$ for each $i\in\{1,2\}$. 
Clearly, each $H_i$  has no $\modfour{0}$-cycles. We also let $H'_1= G-(V(S'_2)\setminus\{y\})$ and $H'_2= G-(V(S'_1)\setminus\{x\})$. By the maximality of $n_1$, each of $H'_1$ and $H'_2$ is non-bipartite, and so each of $H_1$ and $H_2$ is non-bipartite.

\begin{claim}\label{cl:Hi}
$G_1$ is a path of length $3$.  (See Figure~\ref{fig:two:component}.)
\end{claim}

\begin{figure}[ht]
  \centering
\includegraphics[page=11,height=4cm]{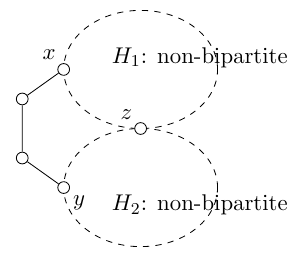}
\caption{An illustration for Claim~\ref{cl:Hi}}\label{fig:two:component}
\end{figure}

\begin{proof}
Take an $(x,y)$-path $Q$ in $G_1$. Let $\ell(Q) \equiv t \pmod{4}$ for some $t\in\{0,1,2,3\}$. Let $G^*_2=(G_2;x,y)\uplus P_{t+1}$. Then $G^*_2$ is $2$-connected, has no $\modfour{0}$-cycles, and $|V(G^*_2)|\le n$. Recall that $xy\not\in E(G)$.

Suppose that $|V(G^*_2)|<n$. Then $e(G^*_2)\le \frac{3|V(G^*_2)|-1}{2}$ by the minimality of $G$. 
Note that $|V(G^*_2)|=n_2+t-1$, and so $e(G^*_2)\le \frac{3|V(G_2^*)|-1}{2}=\frac{3n_2+3t-4}{2}$. 
If $t=0$, then $e(G_2)\le e(G_2^*)+1$ which implies that $e(G_2)\le \frac{3n_2-2}{2}$.
For $t\in\{1,2,3\}$, $e(G_2)\le e(G_2^*)-t \le \frac{3n_2+t-4}{2}$. 
In any case, $e(G_2)\le \frac{3n_2-1}{2}$, which is a contradiction to Lemma~\ref{lem:basic:alpha:4}. 
Thus $|V(G^*_2)|=n$, which implies that $n_1=4$ and so $G_1=P_4$.
\end{proof}

Since $G_1=P_4$ and $H_2$ is non-bipartite, it is clear that $|V(H'_2)|\ge 6$. Let $M_i$ be the smallest subset of $\{0,1,2,3\}$ such that $(x,z)$ is an $M_1$-type in $H_1$ and and $(y,z)$ is an $M_2$-type in $H_2$. Since $H_i$ is non-bipartite,  Proposition~\ref{prop:folklore} implies that $M_i$ is one of $\{0,1\}$, $\{1,2\}$, $\{2,3\}$, and $\{0,3\}$. 
If $(x,z)$ is a $\{1,2\}$-type in $H'_2$, then
$(y,z)$ is a $\{0,1\}$-type in $H'_1$,  which is a contradiction to (i) since each $H'_2$ has more than three vertices. Thus $(x,z)$ is not a $\{1,2\}$-type in $H'_2$. Similarly, $(y,z)$ is not a $\{1,2\}$-type in $H'_1$ by (i).
Thus, $M_i\neq \{2,3\}$ for each $i\in\{1,2\}$.
Note that there is no $(x,y)$-path $P$ in $G_2$ such that $\ell(P) \equiv 1 \pmod{4}$, which implies that $\ell_1+ \ell_2 \not\equiv 1  \pmod 4$ for every $\ell_1 \in M_1$ and $\ell_2 \in M_2$. Then
either $M_1=M_2=\{1,2\}$   or $M_1=M_2=\{0,3\}$ by Table~\ref{table:basic:1201/2303}.
If $M_1=M_2=\{1,2\}$, then $(x,z)$ is a $\{0,1\}$-type in  $H'_1$ and $(y,z)$ is a $\{0,1\}$-type in  $H'_2$, which implies that each of $H_1$ and $H_2$ has three vertices by (i). Then $e(G)\le9 \le \frac{3n-1}{2}$, which is a contradiction. Thus $M_1=M_2=\{0,3\}$.

Let $H^*_i=(H_i;x,z)\uplus(F_3;a,b)$. Then  $|V(H_i^*)|< n$ and $H_i^*$ is a $2$-connected graph that contains no $\modfour{0}$-cycles. By the minimality of $G$,
  $e(H^*_i) \le \frac{3|V(H^*_i)|-1}{2}=\frac{3|V(H_i)|+2}{2}$   and so 
$e(H_i) \le e(H^*_i)-2 \le \frac{3|V(H_i)|-2}{2}$. 
Then $|V(H_1)|+|V(H_2)|=n-1$ and so
\[ e(G) \le e(H_1)+e(H_2)+e(G_1) \le \frac{3(n-1)-4}{2}+3=\frac{3n-1}{2},\] which is a contradiction.
\end{proof}

\begin{lemma}\label{lem:cut:main1} Let $X=\{x,y\}$ be a vertex cut of $G$. There is a connected component $S$ of $G-X$ such that  $G[V(S)\cup X]$ is equal to one of $K_3$, $F_3$, $F_4$  and $F_6$. Moreover, if it is $F_6$, then $x$ and $y$ have a unique common neighbor in $G_{3-i}$. 
\end{lemma}
\begin{proof} 
We let $S_1$ and $S_2$ be the connected components of $G-X$ by Lemma~\ref{lem:threecomp}. Let $G_i=G[V(S_i) \cup X]$, $n_i=|V(G_i)|$ and  $L_i$ be the smallest subset of $\{0,1,2,3\}$ such that $(x,y)$ is an $L_i$-type in $G_i$. 
If $n_i=3$ for some $i\in \{1,2\}$, then it is clear. Suppose that $n_i \ge 4$ for each $i\in \{1,2\}$.
For each $i\in \{1,2\}$, $G_i$ is non-bipartite by Lemma~\ref{lem:threecomp:2}~(ii) and so $L_i$ contains one of $\{0,1\}$, $\{1,2\}$, $\{2,3\}$ and $\{0,3\}$ by Proposition~\ref{prop:folklore}. 
By Table~\ref{table:basic:1201/2303}, we may assume that either $L_1=\{0,1\}$  and $L_2= \{1,2\}$, or $L_1= \{0,3\}$ and $L_2= \{2,3\}$.

If $L_1=\{0,1\}$ and $L_2= \{1,2\}$, then $n_2=3$ by Lemma~\ref{lem:threecomp:2}~(i).
Suppose that $L_1 =\{0, 3\}$ and $L_2 =\{2, 3\}$. 
If $n_2=4$, then $F_4$ is the only possible structure of $G_2$ to have $L_2 =\{2, 3\}$, which is a desired conclusion.  Suppose that $n_2 \ge 5$. Note that $0\in L_1$ and so $n_1\ge 5$.

\begin{claim}\label{clm:n1n2}
$n_2\ge 10$, $n_1\le 6$, and $e(G_1) \le \frac{3n_1-4}{2}$.
\end{claim}
\begin{proof} 
Let $H= (G_1;x,y) \uplus (F_4;a,b)$.
Then $|V(H)|=n_1+2<n$ and $H$ is a $2$-connected graph without $\modfour{0}$-cycles.
By the minimality of $G$, 
\[ e(G_1)= e(H)-4 \le \frac{3(n_1+2)-1}{2}-4 = \frac{3n_1-3}{2}. \]
Suppose to the contrary that $5 \le n_2 \le 9$. 
If $G_2$ is not reversing-equivalent to $F_9$, then $e(G_2) \le \frac{3n_2-4}{2}$ by Proposition~\ref{prop:gadget},  
which contradicts Lemma~\ref{lem:basic:alpha:4}.
Thus $G_2$ is reversing-equivalent to $F_9$.
Let $z$ be the common neighbor of $a$ and $b$ of $F_9$. 
Then $Y = \{x, z\}$ is a vertex cut of $G$ and let $G-Y$ have two connected components $Y_1$ and $Y_2$, say $D_i=G[V(Y_i) \cup Y]$ for each $i\in\{1,2\}$.
Then $(x,z)$ is a $\{1, 2\}$-type in $D_i$ and $(x,z)$ is a $\{0,1\}$-type in $D_j$ for $\{i,j\}=\{1,2\}$. Since both $D_1$ and $D_2$ have more than three vertices, it contradicts Lemma~\ref{lem:threecomp:2}~(i). Thus, $n_2\ge 10$. 

Suppose to the contrary that $n_1 \ge 7$. 
Let $H_1 = (G_1;x,y) \uplus (F_9;a,b)$ and $H_2=(G_2;x,y) \uplus (F_6;a,b)$. For each $i\in\{1,2\}$,  $|V(H_i)|<n$ and $H_i$ is a $2$-connected graph without $\modfour{0}$-cycles.
By the minimality of $G$, 
\[  e(G_1) = e(H_1)-12 \le \frac{3(n_1+7)-1}{2} -12=\frac{3n_1-4}{2}, \quad  e(G_2) = e(H_2)-7 \le \frac{3(n_2+4)-1}{2} -7=\frac{3n_2-3}{2},\]
which contradicts Lemma~\ref{lem:basic:alpha:4}.  
Therefore, $n_1\le 6$. By Proposition~\ref{prop:gadget}, $e(G_1) \le \frac{3n_1-4}{2}$.
\end{proof}

\begin{claim}\label{clm:G_2^*}
$x$ and $y$ have a unique common neighbor in $G_2$, and $e(G_2) \le \frac{3n_2-2}{2}$.
\end{claim}
\begin{proof} 
Let $G_2^*$ be the graph obtained from $G_2$ by identifying $x$ and $y$, where $v^*$ is the identified vertex of $x$ and $y$ in $G_2^*$.  Then $G_2^*$ has no $\modfour{0}$-cycles, since $(x,y)$ is a $\{2, 3\}$-type in $G_2$. 
If $G_2^*$ is not $2$-connected,
then the vertex $v^*$ is a cut-vertex of $G_2^*$, which implies that $G-X$ has at least three connected components, and it is a contradiction to Lemma~\ref{lem:threecomp}. Therefore $G_2^*$ is $2$-connected.

Note that $|V(G_2^*)|<n$ and $G^*_2$ is a $2$-connected graph without $\modfour{0}$-cycles. By the minimality of $G$, $e(G_2^*) \le \frac{3(n_2-1)-1}{2}$. 
If $x$ and $y$ have no common neighbor in $G_2$, then $e(G_2)=e(G^*_2) \le \frac{3n_2-4}{2}$, which contradicts Lemma~\ref{lem:basic:alpha:4} since $e(G_1) \le \frac{3n_1-4}{2}$ by Claim~\ref{clm:n1n2}. 
Thus $x$ and $y$ have a common neighbor $z$ in $G_2$. Then such $z$ is unique since $G_2$ has no $4$-cycles.
Then $e(G_2)=e(G^*_2)+1 \le \frac{3n_2-2}{2}$.
\end{proof}

By Claims~\ref{clm:n1n2} and \ref{clm:G_2^*}, 
$e(G_1) \le \frac{3n_1-4}{2}$ and $e(G_2) \le \frac{3n_2-2}{2}$.
By Lemma~\ref{lem:basic:alpha:4}, 
$e(G_1)= \frac{3n_1-4}{2}$ and $ e(G_2)=\frac{3n_2-2}{2}$. By Proposition~\ref{prop:gadget},
$G_1$ is reversing-equivalent to $F_6$. Since any graph reversing-equivalent to $F_6$ is isomorphic to $F_6$, $G_1$ is isomorphic to $F_6$
\end{proof}

Now we are ready to prove Lemma~\ref{lem:no_F_6}.

\begin{proof}[Proof of Lemma~\ref{lem:no_F_6}]
By Lemmas~\ref{lem:threecomp} and~\ref{lem:cut:main1}, it suffices to show that each $G_i$ is not equal to $F_6$, where
$S_1$ and $S_2$ are the connected components of $G-X$ and $G_i=G[V(S_i) \cup X]$. 
Suppose to the contrary that $G_1$ is equal to $F_6$. Then $x$ and $y$ have a common neighbor $z$ in $G_2$ by Lemma~\ref{lem:cut:main1}. 
Let $n_i=|V(G_i)|$. Recall that $(x,y)$ is a $\{0,3\}$-type in $G_1$ and a $\{2,3\}$-type in $G_2$.  Moreover, $\{x,y\}=\{a,b\}$ where $a$ and $b$ are two vertices in Figure~\ref{fig:gadget}.
For an $(x,y)$-path $P$ in $G_2-z$ and a $(V(P),z)$-path $Q$ in $G_2-\{x,y\}$, we call this pair $(P,Q)$ a \textit{nice pair}. 

First, we will show that a nice pair exists. Suppose to the contrary that $G_2-z$ is not connected. Then each of $Z_1=\{x,z\}$ and $Z_2=\{y,z\}$ is a vertex cut of $G$. There are a connected component $S_1'$ of $G-Z_1$ not containing $y$ and  a connected component $S_2'$ of $G-Z_2$ not containing $x$. Let $H_i=G[V(S_i')\cup Z_i]$. Since $(x,z)$ and $(y,z)$ is a $\{1, 2\}$-type in $H_i$, $K_3$ is the only possible case where two vertices of a vertex cut are adjacent by Lemma~\ref{lem:cut:main1}, which implies that $H_i=K_3$ for each $i\in \{1,2\}$. It is a contradiction to the fact that $(x,y)$ is a $\{2,3\}$-type in $G_2$.
Thus, $G_2-z$ is connected. Then there is an $(x,y)$-path in $G_2-z$. 
Note that $S_2=G_2-\{x,y\}$ is a connected component.
For every $(x,y)$-path $P$ in $G_2-z$, $V(P)\setminus\{x,y\}\neq \emptyset$ and so there is a $(V(P),z)$-path $Q$ in $G_2-\{x,y\}$.
Thus a nice pair always exists.

\begin{claim}\label{clm:pqmod4:1}
Let $(P,Q)$ be a nice pair, say $Q$ is a $(q,z)$-path. Let $R$  be a $(V(Q), V(P))$-path in $G_2-\{z,q\}$, say $R$ is  an $(r,r')$-path.
Then the following holds. 
\begin{itemize}
\item[\rm(i)]  $\ell(P)\equiv 3 \pmod{4}$ and $\ell(Q)$ is even.   
\item[\rm(ii)] $r'\in \{x,y\}$, the subpath of $P$ between $q$ and $r'$ is an odd path, and  $\ell(R)$ is odd. 
\end{itemize}
\end{claim}

\begin{proof} (i) 
If $\ell(P)  \equiv 2 \pmod{4}$, then $P+yzx$ is a $\modfour{0}$-cycle,  a contradiction. 
Thus $\ell(P)  \equiv 3 \pmod{4}$, since $(x,y)$ is a $\{2,3\}$-type in $G_2$.
We may assume that $\ell(P[x,q])$ is even. Then $\ell(P[q,y])$ is odd. 
Note that there is an even $(x,q)$-path in $G$ using  an $(x,y)$-path of length three in $G_1$ and $P[q,y]$. 
Since there is no even $(x,q)$-theta graph by Lemma~\ref{thm:planar}~(i), $Q+zx$ is an odd path, and so $\ell(Q)$ is even.

\medskip

\noindent (ii) By (i), we may assume $P[x,q]$ is an even path and denote it by $P_e$. We denote $P[q,y]$ by $P_o$. Let $Q_1=Q[q,r]$ and $Q_2=Q[r,z]$. We define an $(x,y)$-path $P^*$ in $G_2-z$ as follows (see Figure~\ref{fig:noF6:cases}): If $r'\in V(P_o)$, then let $P^*: P_e+Q_1+R+P_o[r',y]$.  If $r'\in V(P_e)$, let $P^*:P_e[x,r']+\overleftarrow{R}+\overleftarrow{Q_1}+P_o$.  
Since $(P^*,Q_2)$ is a nice pair, $\ell(Q_2)$ is even by (i), and so $\ell(Q_1)$ is even.
\begin{figure}[!ht]
\centering
\includegraphics[height=4cm,page=6]{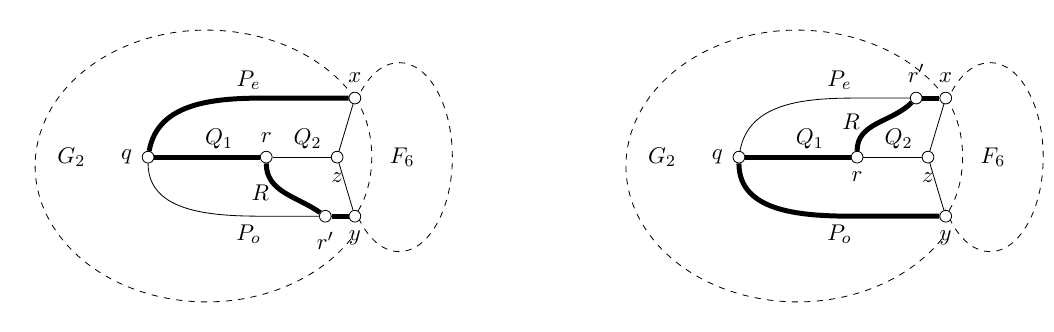}
\caption{Illustrations for Claim~\ref{clm:pqmod4:1}, where  thick lines show the path $P^*$}\label{fig:noF6:cases}
\end{figure}

Suppose to the contrary that $r'\neq y$.
If $r'\neq x$, then 
$(P,\overleftarrow{R}+Q_2)$ is a nice pair, which implies that $\ell(\overleftarrow{R}+Q_2)$ is even by (i) and so  $\ell(R)$ is even. 
If $r'=x$, then comparing the lengths of $P$ and $P^*$, $\ell(R)+\ell(Q_1)\equiv \ell(P_e)\pmod{4}$ and therefore $\ell(R)$ is even.
Hence, in any case, $\ell(R)$ is even and so $Q_1+R$ is an even path.
Comparing the lengths of $P$ and $P^*$, it follows that for the subpath $P'$ of $P$ between $r'$ and $q$, $\ell(P') \equiv \ell(R)+\ell(Q_1)\pmod{4}$  by (i). Then $Q_1+R$ and $P'$ form a $\modfour{0}$-cycle, which is a contradiction.
Therefore $r'=y$.  Moreover, by (i), $\ell(R)+\ell(Q_1)\equiv \ell(P_o) \pmod{4}$  and so $\ell(R)$ is odd.
\end{proof}

\begin{claim}\label{clm:end:Q:1}
There is a nice pair $(P,Q)$ such that $\ell(Q)=2$ and the ends of $Q$ form a vertex cut of $G$.
\end{claim}

\begin{proof} 
Suppose to the contrary that for every nice pair $(P,Q)$, the ends of $Q$ do not form a vertex cut of $G$. We take such $P$ and $Q$ so that
$\ell(Q)$ is as small as possible. Say $Q$ is a $(q,z)$-path. We may assume $P[x,q]$ is an even path and $P[q,y]$ is an odd path.
Since $\{q,z\}$ is not a vertex cut of $G$,
there is a $(V(Q), V(P))$-path $R$ in $G_2-\{q,z\}$.
Say $R$ is an $(r,r')$-path. 
By Claim~\ref{clm:pqmod4:1} (ii), $r'=y$. Then it is a contradiction to the choice of $Q$ by a nice pair $(P[x,q]+Q[q,r]+R, Q[r,z])$.
Therefore, for some nice pair $(P,Q)$, the ends of $Q$ form a vertex cut of $G$. By Lemma~\ref{lem:cut:main1} and Claim~\ref{clm:pqmod4:1} (i), $Q$ is a path of length two.
\end{proof}

We take a nice pair $(P,Q)$ satisfying Claim~\ref{clm:end:Q:1}. Say $Q$ is a $(q,z)$-path. Without loss of generality, by Claim~\ref{clm:pqmod4:1}~(i), we may assume $P[x,q]$ is an even path and denote it by $P_e$. We denote $P[q,y]$ by $P_o$. 

Suppose to the contrary that $Y=\{y,q\}$ is a vertex cut of $G$. 
By Lemma~\ref{lem:threecomp}, $G-Y$ has exactly two connected components, and 
the connected component of $G-Y$ containing the vertex $z$ has more than 6 vertices. Thus for the other connected component $D$, $H=G[V(D)\cup Y]$ contains $P[q,y]$ and it is one of $K_3$, $F_3$, $F_4$, and $F_6$ by Lemma~\ref{lem:cut:main1}.
It is a contradiction, since every $(q,y)$-path in $H$ has odd length modulo $4$. Therefore, there is a $(V(P_o), V(P_e)\cup V(Q))$-path $Q'$ in $G-\{y,q\}$, say $Q'$ is a $(q',t')$-path. 
Note that $\ell(Q)=2$ and so $t'\not\in V(Q)\setminus\{z\}$ by Claim~\ref{clm:end:Q:1}.

\begin{claim}\label{clm:Q1:1}
$(P,Q')$ is a nice pair, and $C: Q'+zy+\overleftarrow{P}[y,q']$ is an odd cycle.
\end{claim}
\begin{proof}
If $t'\in V(P_e)\setminus\{q,x\}$, then 
$(P[x,t']+\overleftarrow{Q'}+P[q',y], P[t',q]+Q)$ is a nice pair and therefore $R:P[q,q']$ violates Claim~\ref{clm:pqmod4:1}~(ii), since $q'\not\in\{x,y\}$. Thus $t' \in \{x,z\}$. 
Suppose to the contrary that $t'=x$.
Then $(\overleftarrow{Q'}+P[q',y], \overleftarrow{P}[q',q]+Q)$ is a nice pair.
Then a path $R:\overleftarrow{P_e}$  violates Claim~\ref{clm:pqmod4:1}~(ii), since $\ell(R)$ is even. Thus $t'=z$ and so $(P,Q')$ is a nice pair. 
Then $\ell(Q')$ is even by Claim~\ref{clm:pqmod4:1}~(i), which also implies that $\ell(P[q,q'])$ is odd, otherwise $Q'$, $P[q',y]+yz$, and $\overleftarrow{P}[q',q]+Q$ make an even theta graph. Then 
 $\ell(P[q',y])$ is even, which implies that $C: Q'+zy+\overleftarrow{P}[y,q']$ is an odd cycle.
\end{proof}
By Claim~\ref{clm:Q1:1}, $C: Q'+zy+\overleftarrow{P}[y,q']$ is an odd cycle.
Note that since $F_6$ has a $3$-cycle and $5$-cycle, $G_1$ contains an odd cycle $C_0$ such that $\ell(C)\equiv \ell(C_0) \pmod{4}$. 
Consider a reverse $H$ of $G$ at $(x,y)$ with $G_1\setminus \{x,y\}$ so that $C$ and $C_0$ share a vertex $y$ in $G$ or $H$. It is clear that $H$ is also a minimal counterexample. Then there is a $(V(C_0),V(C))$-path of even length in $G$ or $H$ using $P_e+Q$ or an edge $xz$, which is a contradiction to Lemma~\ref{lem:previous}~(i).
 \end{proof}

\subsection{Proof of Lemma~\ref{lem:main}}

\begin{lemma}\label{lem:3face}
There are at most two triangles in $G$.  
\end{lemma}

\begin{proof}
Suppose to the contrary that $G$ has three  triangles $T_1$, $T_2$, and $T_3$.
Let $V(T_i)=\{p_i,q_i,t_i\}$ for each $i\in \{1,2,3\}$.
By Proposition~\ref{clm:triperm},  we may assume that there are two vertex-disjoint $(p_1,p_2)$-path $P$ and $(q_1,q_2)$-path $Q$ such that $p_3t_3\in E(P)$ and $q_3\in V(Q)$. We may assume that $p_3$ is closer to $p_1$ than $t_3$ along the path $P$. 
For simplicity, we define the following (see Figure~\ref{figure:triangle}):
\[\begin{array}{lll}
A_1 = V(P[p_1, p_3])-\{p_3\},\quad & A_2 = V(Q[q_1, q_3])-\{q_3\}, \quad&  A= \{t_1\} \cup A_1\cup A_2,\\
B_1 = V(P[t_3, p_2])-\{t_3\},\quad & B_2 = V(Q[q_3, q_2])-\{q_3\}, \quad&  B= \{t_2,t_3\} \cup B_1 \cup B_2.
\end{array}\] Note that $A_i$ or $B_i$ might be the empty set if a vertex of $T_3$ coincides with a vertex of $T_1$ or $T_2$. (For example, $A_1= \emptyset$ if $p_1 = p_3$.)
\begin{figure}[!ht]
\centering
  \includegraphics[page=7]{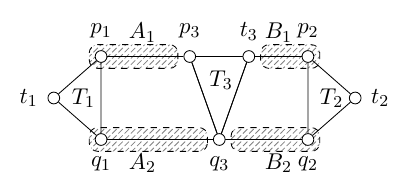}
\caption{An illustration for Lemma~\ref{lem:3face}, where the shaded parts represent $A_1$, $A_2$, $B_1$, and $B_2$.}\label{figure:triangle}
        \end{figure}

Suppose that $\{p_3, q_3\}$ is a vertex cut of $G$. 
Then $G-\{p_3, q_3\}$ has exactly two components $S_1$ and $S_2$ by  Lemma~\ref{lem:no_F_6}, one of which induces $K_3$ together with $\{p_3,q_3\}$, since $K_3$ is the only graph that two vertices of a vertex cut are adjacent. But we can observe that $G[V(S_i)\cup\{p_2,q_3\}]$ has at least four vertices, since two triangles share at most one vertex.
Therefore $\{p_3, q_3\}$ is not a vertex cut of $G$.
Hence, there is an $(A,B)$-path $R$ in $G-\{p_3,q_3\}$, say $R$ is an $(a,b)$-path. 
We divide the proof into following 6 cases, and one can check the following cases cover all possibilities:
\begin{itemize}
\item[(1)] $a\in A_1$ and $b\in B_1\cup \{t_3\}$;
\item[(2)] $a\in A_2$ and $b\in B_2$;
\item[(3)] $a\in A_2$ and $b\in B_{1}$ (The case where $a\in A_1$ and $b\in B_2$ is similar.);
\item[(4)]  $a=t_1$ and $b\in B_1\cup B_2$ 
(The case where $a\in A_1\cup A_2$ and $b=t_2$ is similar.);
\item[(5)] $a=t_1$, $b\in\{t_2,t_3\}$; 
\item[(6)] $a\in A_2$, $b=t_3$.
\end{itemize}
In each of the cases (1)$\sim$(3), we reach a contradiction since it is not difficult to find a cycle $C$ or a theta graph $\Theta$ described in Lemma~\ref{lem:odd:cycles}~(i) or (ii), where  
Figure~\ref{fig:proof:triangle} shows such $C$ or $\Theta$ with thick lines. 
For the case (1), the thick lines show a theta graph $\Theta$ described in Lemma~\ref{lem:odd:cycles}~(ii). For the cases (2) and (3),  the thick lines show  a cycle $C$ described in Lemma~\ref{lem:odd:cycles}~(i).
Note that in each case, the vertices of $T_1$, $T_2$, $T_3$ not on $C$ or $\Theta$ are distinct. 

\begin{figure}[!ht]
\centering
 \includegraphics[page=8,height=3.3cm]{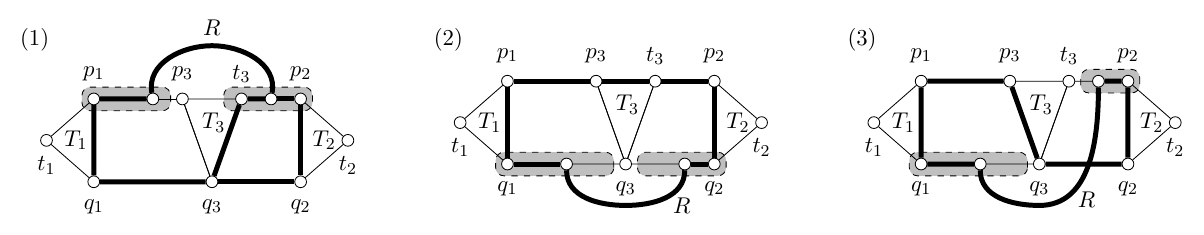}
\caption{Illustrations for the cases (1)$\sim$(3), where the shaded parts represent the parts containing the end vertices of $R$.}\label{fig:proof:triangle}
\end{figure}

For the case (5), if $b=t_2$, then $P$, $Q$, $ R$  are  three vertex-disjoint paths between $V(T_1)$ and $V(T_2)$, and, if $b=t_3$, then $P[p_1,p_3]$, $Q[q_1,q_3]$, $R$ are  three vertex-disjoint paths between $V(T_1)$ and $V(T_3)$. We reach a contradiction to Lemma~\ref{lem:odd:cycles}~(iii).
 
For the case (4), we can find three vertex-disjoint path between $V(T_1)$ and $V(T_3)$, two of them are $P[p_1,p_3]$ and $Q[q_1,q_3]$ and the other one is the path with thick lines in first two figures of Figure~\ref{fig:proof:triangle:456}.

\begin{figure}[!ht]
\centering
 \includegraphics[page=9,height=3.3cm]{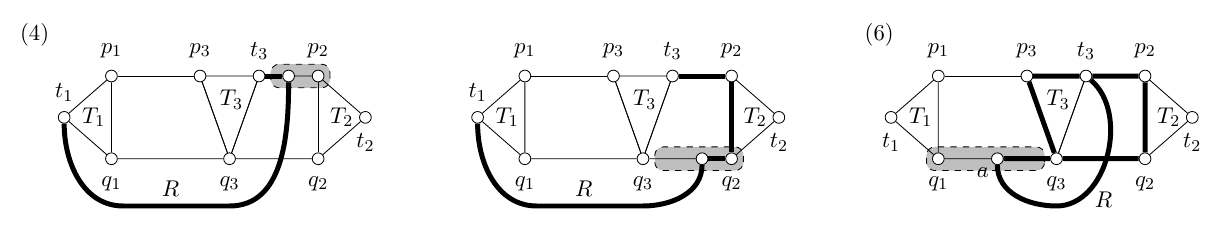}
\caption{Illustrations for the cases (4) and (6), where the shaded part represent the part containing an end vertex of $R$}\label{fig:proof:triangle:456}
\end{figure}

For the case (6), that is, $a\in  A_2$ and $t_3=b$, we will reach a contradiction to Lemma~\ref{lem:previous}~(i) by showing  that there is an even $(t_3,q_3)$-theta graph. 
See the last figure of Figure~\ref{fig:proof:triangle:456}.
It is clear that $t_3p_3q_3$ is a path of length two.
By Lemma~\ref{lem:odd:cycles}~(iii), $P[t_3,p_2]+p_2q_2+\overleftarrow{Q}[q_2,q_3]$ is a $\modfour{0}$-path, and both $\overleftarrow{P}[p_3,p_1]+p_1q_1+Q[q_1,a]+Q[a,q_3]$ and $\overleftarrow{P}[p_3,p_1]+p_1q_1+Q[q_1,a]+R$ are $\modfour{0}$-paths. Therefore $Q[a,q_3]$ and $R$ have the same length modulo $4$, and so $\overleftarrow{R}+Q[a,q_3]$ is an even $(t_3,q_3)$-path.
Hence, the thick lines in the last figure of Figure~\ref{fig:proof:triangle:456} shows an even $(t_3,q_3)$-theta graph.
It completes the proof.
\end{proof}

Some part of the proof is identical to the proof in \cite{0mod4cycle2025}, but included for the sake of completeness.
 
\begin{lemma}\label{lem:55face}
There are at most five $5$-faces in $G$.  
\end{lemma}

\begin{proof}
Suppose to the contrary that there are two vertex-disjoint $5$-cycles $C_1$ and $C_2$. 
Let $X$ be a smallest $(V(C_1),V(C_2))$-separating set.
Then $|X| \ge 2$ since $G$ is $2$-connected. If $|X| \ge 3$, then by Theorem~\ref{thm:menger}, there are three vertex-disjoint $(V(C_1),V(C_2))$-paths, which is a contradiction to Lemma~\ref{lem:previous}~(ii). 
Thus $|X|=2$.
By Lemma~\ref{lem:threecomp}, $G-X$ has exactly two connected components $S_1$ and $S_2$.
Then each of $G[V(S_1) \cup X]$ and $G[V(S_2) \cup X]$ has either $C_1$ or $C_2$ as a subgraph, which is a contradiction to Lemma~\ref{lem:no_F_6} since there is no $5$-cycle in $K_3$, $F_3$, and $F_4$. 
Therefore $G$ has no vertex-disjoint $5$-cycles.
Since $G$ has no $\modfour{0}$-cycles, every two $5$-cycles intersect at a vertex or a path of length $2$.  

\begin{claim}\label{clm:5faceatv}
There are no three $5$-faces such that every two of them intersect at a vertex.
\end{claim}
\begin{proof}
Suppose that there are three $5$-faces $C_1$, $C_2$, {and} $C_3$ such that every two of them intersect at a vertex.
By Lemma~\ref{lem:odd:cycles}~(iv), $V(C_1) \cap V(C_2) \cap V(C_3)=\{v\}$ for some vertex $v$. 
Let $B_i=C_i-v$ for each $i\in\{1,2,3\}$.
Suppose to the contrary that $V(C_k)$ separates $B_i$ and $B_j$ for some $i,j,k$. We may assume $i=1$, $j=2$, and $k=3$. Since $G$ is a plane graph and $C_3$ is a face, 
there is a vertex $u$ such that by letting $R_1= C_3[v,u]$ and $R_2=C_3[u,v]$,  
for every connected component $H$ of  $G-(V(C_1)\cup V(C_2) \cup V(C_3))$, {either} $H$  does not have a neighbor in $V(R_2)\setminus\{u,v\}$  or
$H$ does not have a neighbor in $V(R_1)\setminus\{u,v\}$.
Therefore $X = \{u,v\}$ is a vertex cut of $G$. 
By Lemma~\ref{lem:no_F_6}, $G-X$ has exactly two connected components. For each connected component $S$ of $G-X$, $G[S \cup X]$ contains $C_1$ or $C_2$ and so it has more than $5$ vertices, which contradicts Lemma~\ref{lem:no_F_6}.

Thus, there is a $(V(B_i), V(B_{i+1}))$-path $Q_i$ in $G-V(C_{i+2})$ for each $i\in\{1,2,3\}$.
By Lemma~\ref{lem:previous}~(iii), we may assume that $Q_1$ and $Q_2$ intersect at a vertex. Then $Q_1\cup Q_2$ makes a connected subgraph $H$ of $G-v$ such that $V(H)\cap V(C_i)\neq \emptyset$ and 
$V(H)\cap E(C_i)=\emptyset$ for each $i\in \{1,2,3\}$, which is a contradiction to Lemma~\ref{lem:odd:cycles}~(iv).
\end{proof}

\begin{claim}\label{clm:5fat3p}
There are no three $5$-faces such that every two of them intersect at a path of length $2$. 
\end{claim}

\begin{proof}
Suppose that there are three $5$-faces $C_1$, $C_2$, and $C_3$ such that every two of them intersect at a path of length $2$. 
Let $uvw$ be the path of length $2$ in which $C_1$ and $C_2$ intersect. We let $C_1:uvw w_1u_1u$ and $C_2:uvw w_2u_2u$. 
Since $C_1$ and $C_3$ intersect at a path of length $2$, without loss of generality, we may assume that $w \in V(C_3)$.
Then it holds that $C_3 : ww_1u_1u_2w_2w$.
See Figure~\ref{fig:fface} for an illustration.

\begin{figure}[!ht] 
\centering
\includegraphics[page=13, height=4.2cm]{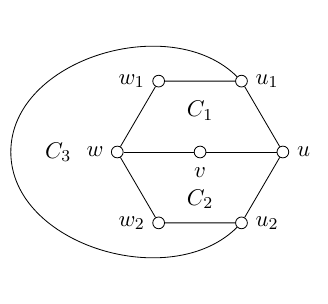}
\caption{An illustration for Claim~\ref{clm:5fat3p}.}\label{fig:fface}
\end{figure} 
Note that $G$ is a plane graph and $C_1$, $C_2$, $C_3$ are faces. Then 
$G_0 = G - \{v, w, w_1, w_2\}$ is $2$-connected. By the minimality of $G$, $e(G_0) \le \frac{3(n-4)-1}{2}$. Then $e(G) = e(G_0)+6 \le \frac{3(n-4)-1}{2}+6 = \frac{3n-1}{2}$, which is a contradiction. 
\end{proof}

Recall that every two $5$-cycles intersect at a vertex or a path of length $2$.  
Note that the Ramsey number $r(3,3) = 6$. 
If $G$ has at least six $5$-faces, then three of them either pairwise intersect at a vertex or pairwise intersect at a path of length two, contradicting Claim~\ref{clm:5faceatv} or Claim~\ref{clm:5fat3p}. 
\end{proof}

By Lemmas~\ref{lem:3face}~and~\ref{lem:55face}, Lemma~\ref{lem:main} holds.

\section{Extremal graphs}

In this section, we discuss how to construct extremal $2$-connected graphs without $\modfour{0}$-cycles. 
We define the \textit{gap} function $gap(G)$ by
$gap(G)= (3n-1)-2e(G)$ for an $n$-vertex graph $G$.
Note that our main theorem says that if $G$ has no $\modfour{0}$-cycles, then $gap(G)\ge 0$.

\begin{proposition}\label{prop:gap:Fi}
Let $G'$ be the graph defined by $(G;x,y)\uplus (F_i;a,b)$ for a graph $(G; x,y)$ and a graph $(F_i;a,b)$ in Figure~\ref{fig:gadget}.
Then $gap(G') < gap(G)$. More precisely,
\[ gap(G')=\begin{cases}
  gap(G)-1 & \text{if }F_i=F_3 \\
  gap(G)-2 &  \text{if }F_i\in \{F_4, F_6 \} \\
  gap(G)-3 &   \text{if }F_i=F_9,
\end{cases}     \qquad\text{and}\qquad  gap(G') \le \begin{cases}
  gap(G)-1 & \text{if }F_i=F_7 \\
  gap(G)-2 &  \text{if }F_i=F_8. 
\end{cases}   \]
\end{proposition}

\begin{proof} Let $|V(G')|-|V(G)|=d_v$ and $e(G')-e(G)=d_e$. Then $gap(G')=gap(G)+3d_v-2d_e$.
If there is no edge between $a$ and $b$ in $F_i$, then $3d_v-2d_e=3|V(F_i)|-2e(F_i)-6$.
If there is an edge between $a$ and $b$ in both $G$ and $F_i$, then the edge might not be counted in $d_e$ and so 
$3|V(F_i)|-2e(F_i)-6\le 3d_v-2d_e \le 3|V(F_i)|-2e(F_i)-4$. Thus the proposition holds.
\end{proof}

\begin{proposition}\label{prop:gap:path}
For a graph $(G; x,y)$,
let $G'$ be the graph defined by $(G;x,y)\uplus P_4$  or  $(G;x,y)\uplus (F_6;b,c)$, where $c$ is a neighbor of $b$, and let $G''$ be the graph defined by $(G;x,y)\uplus K_3$.
Then $gap(G') = gap(G)$ and $gap(G'')\le gap(G)-1$.
\end{proposition}

\begin{proof}
If $G'=(G;x,y)\uplus P_4$, then 
$|V(G')|-|V(G)|=2$ and $e(G')-e(G)=3$, and so $gap(G')=gap(G)+0$.
If $G'= (G;x,y)\uplus (F_6;b,c)$, then 
$|V(G')|-|V(G)|=4$ and $e(G')-e(G)=6$, and so $gap(G')=gap(G)+0$.
For  $G''=(G;x,y)\uplus K_3$, 
$|V(G'')|-|V(G)|=1$ and $e(G'')-e(G)\ge 2$, and so $gap(G')\le gap(G)-1$.
\end{proof}

For simplicity, we let $L_{i}=\{i,i+1\}$ for each $i\in\{0,1,2\}$, and $L_3=\{0,3\}$. We say $L_0$ and $L_1$ are {\textit{matched}} to each other, and $L_2$ and $L_3$ are {\textit{matched}} to each other. Recall that $(a,b)$ is an $L_3$-type in $F_6$, an  $L_0$-type in $F_7$, an $L_1$-type in $F_8$, and an $L_2$-type in $F_4$ and $F_9$. Moreover, if $c$ is a neighbor of $b$ in $F_6$, then $(b,c)$ is an $L_0$-type in $F_6$.

Let $\sigma: G_0, \ldots, G_k$ be a sequence of graphs without $\modfour{0}$-cycles. We call $\sigma$ a \textit{gap-reducing sequence} if $gap(G_s)\ge gap(G_{s+1})$ for each $0\le s\le k-1$, and $\sigma$ is said to be \textit{strict} if $gap(G_s)> gap(G_{s+1})$ for each $0\le s\le k-1$.

Now we explain a way to obtain a (strict) gap-reducing sequence  $\sigma: G_0,\ldots,G_k$ of  graphs without $\modfour{0}$-cycles. We take a graph $G_0$
without $\modfour{0}$-cycles. Then $G_{s+1}$ is obtained from $G_s$ by one of the following procedures (R1), (R2), and (R3).
\begin{itemize}
\item[(R1)] Find two vertices $x$ and $y$ such that $(x,y)$ is an  $L_i$-type in $G$ for some $i\in \{0,1,2,3\}$. Let $L_j$ be the set matched to $L_i$.
Then we define $G_{s+1}$ by $(G_s;x,y)\uplus (H;a',b')$, where $(a',b')$ is an $L_j$-type in $H$ and either $H=K_3$ or $(H;a',b')=(F_t;a,b)$ in Figure~\ref{fig:gadget}. 
\item[(R2)] Find two vertices $x$ and $y$ such that 
$(x,y)$ is an $L_1$-type in $G$. 
Then we define {$G_{s+1}$} by $(G_s;x,y)\uplus (F_6;b,c)$, where $c$ is a neighbor of $b$ in Figure~\ref{fig:gadget}.
\item[(R3)] Find two vertices $x$ and $y$ such that 
 $x$ and $y$ are not connected by a $\modfour{1}$-path in $G$. 
Then we define {$G_{s+1}$} by $(G_s;x,y)\uplus P_{4}$.
\end{itemize}
By Propositions~\ref{prop:gap:Fi}~and~\ref{prop:gap:path}, it is clear that every sequence obtained from this way is a gap-reducing sequence of graphs without $\modfour{0}$-cycles. If we conduct only procedure (R1), then the resulting sequence is a strict gap-reducing sequence. In addition, if $G_0$ is $2$-connected, then every graph in the sequence is also $2$-connected. Figure~\ref{fig:tight:exmaple} shows an example, which starts with a $5$-cycle $G_0$. 

\begin{figure}[ht]
\centering 
\includegraphics[page=12,width=16cm]{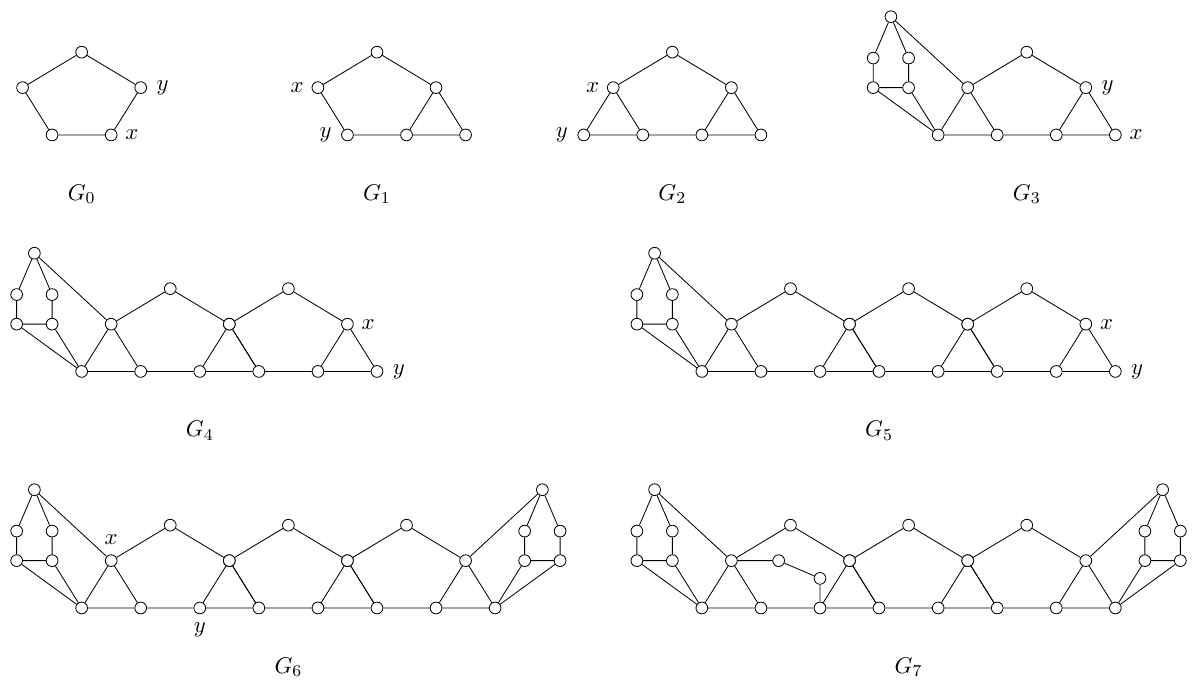}\\
\centering
\footnotesize{\begin{tabular}
{c||c c c c c c c c c c c c c c c c c}
$n$&$G_0$&\!\!$\rightarrow\!\!$ &$G_1$ &$\!\!\rightarrow\!\!$ &$G_2$&$\!\!\rightarrow\!\!$ &$G_3$&$\!\!\rightarrow\!\!$ &$G_4$&$\!\!\rightarrow\!\!$ &$G_5$&$\rightarrow\!\!$ &$G_6$ &$\!\!\rightarrow\!\!$ &$G_7$\\ \hline    \hline  
$gap(G_i)$& 4&\!\!(R1)\!\!& 3 &\!\!(R1)\!\!&2 &\!\!(R1)\!\!&1 &\!\!(R2)\!\!&1&\!\!(R2)\!\!&1&\!\!(R1)\!\!&0&\!\!(R3)\!\!&0&  \\ 
\end{tabular} }\\ 
\caption{A gap-reducing sequence $\sigma: G_0, G_1, \ldots,G_7$ by the procedures (R1), (R2), and (R3)}\label{fig:tight:exmaple} 
\end{figure}
 
As one can observe from the example in Figure~\ref{fig:tight:exmaple}, we can construct infinitely many extremal $2$-connected graphs without $\modfour{0}$-cycles. In addition to the procedures (R2) and (R3), we can find other graph constructions that preserve the value of the gap function, and these can be used to produce extremal examples with diverse structures.

Here are some other examples. See Figure~\ref{fig2:tight:exmaple2} for an illustration. Consider a graph $G$ obtained from $G_0=(F_6;a,b)$ by applying the procedure (R1) with $(F_4;a,b)$ once and then by applying the procedure (R3) with $k$ times at the vertices $a$ and $b$. Then the resulting graph $G$ has $8+2k$ vertices and $11+3k$ edges. 
We also consider a graph $H$ obtained from $G_0=K_3$ by applying the procedure (R1) with $(F_7;a,b)$ twice at different edges of $K_3$, and then by applying the procedure (R3) with $k$ times at the vertices $a$ and $b$. Then the resulting graph $H$ has $13+2k$ vertices and $19+3k$ edges. 
This implies that for every $n\ge 12$, there is a $2$-connected $n$-vertex $G$ without $\modfour{0}$-cycles
such that $e(G)= \lfloor \frac{3n-1}{2} \rfloor$.

\begin{figure}[ht]
\centering 
\includegraphics[page=14,width=12.5cm]{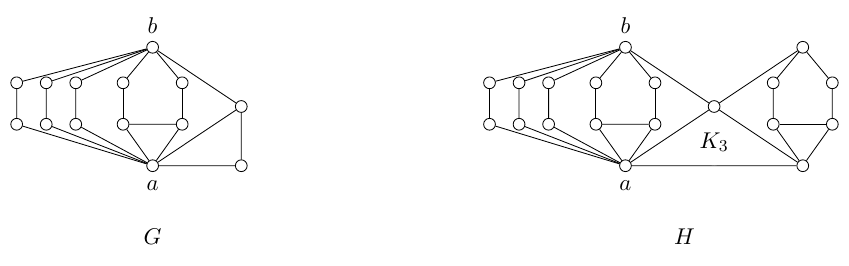}\\
\caption{Tight examples}
\label{fig2:tight:exmaple2} 
\end{figure}

In \cite{0mod4cycle2025}, the authors construct an $n$-vertex graph $G$ 
without  $\modfour{0}$-cycles such that 
$e(G)=\left\lfloor \frac{19(n -1)}{12} \right\rfloor$ as follows. 
They also denote some graphs reversing-equivalent to the graphs $F_8$ and $(F_7; a,b) \uplus (F_8; a,b)$ by $L_8$ and $L_{13}$, respectively. Extremal examples for general graphs are obtained by performing one-vertex identifications on multiple copies of  {$L_{13}$, $L_8$, $K_3$, and $K_2$}. According to our analysis in this paper, their extremal examples are essential in a certain sense.
Let $G$ be an $n$-vertex graph with exactly $\left\lfloor \frac{19(n -1)}{12} \right\rfloor$ edges that contains no $\modfour{0}$-cycles.
If $n$ is sufficiently large, then $
\left\lfloor \frac{19(n -1)}{12} \right\rfloor > \left\lfloor \frac{3n - 1}{2} \right\rfloor$,  
which implies that no block of $G$ can contain a large number of vertices.

\section*{Acknowledgements} 
Hojin Chu was supported by a KIAS Individual Grant (CG101801) at Korea Institute for Advanced Study.
Boram Park and Homoon Ryu were supported by the National Research Foundation of Korea(NRF) grant funded by the Korea government(MSIT) (No. RS-2025-00523206).

\section*{Appendix
%\footnote{In accordance with the reviewer's suggestion, the authors are open to either providing a more detailed version of the proof or omitting it entirely.} 
}
We provide a sketch of the proof of Proposition~\ref{prop:gadget}. Some cases are written in a rough way, as they are straightforward yet tedious to check. Due to the large number of such cases, verification is provided via code.
See the supplementary repository: \url{https://github.com/Homoon-ryu/2con_0mod4-cycle_free}.
 
\begin{proof}[Proof of Proposition~\ref{prop:gadget}]
For simplicity, let $t(n)$ be the target function defined as follows:
\[t(n)=\begin{cases}
\frac{3n-4}{2} & \text{if } L=\{0,3\} \text{ and }  n \le 6;\\
\frac{3n-3}{2} & \text{if } L=\{0,1\} \text{ and }   n \le 7;\\
\frac{3n-2}{2} & \text{if } L=\{1,2\} \text{ and }   n \le 8;\\
\frac{3n-3}{2} & \text{if } L=\{2,3\} \text{ and }   n \le 9.
\end{cases}
\]

We have the following Table~\ref{table1}, where each cell in the last four rows gives the target upper bound for $e(H)$. 
That is, if $t(n)$ is an integer, each cell gives $t(n)-1$ except for the cases of equality parts, and each cell has a value for $t(n)$ otherwise.

\renewcommand{\arraystretch}{1.2}
\begin{table}[!ht] 
\centering\begin{tabular}
{c||C{0.7cm}|C{0.7cm}|C{0.7cm}|c|c|c|c|c}
$n$&3&4 &5&6 &7&8&9 \\ \hline \hline 
$\left\lfloor \frac{19(n -1)}{12}\right\rfloor$& 3&4 &6&7 &9&11&12    \\ \hline \hline 
 $L=\{0,3\}$& $\frac{5}{2}$  & 3 &$\frac{11}{2}$& 7  {\bf{[$F_6$]}} &   \cellcolor[gray]{0.9}&   \cellcolor[gray]{0.9}&   \cellcolor[gray]{0.9}   \\ \hline
 $L=\{0,1\}$& 2 &$\frac{9}{2}$ & 5 &$\frac{15}{2}$ & 9  {\bf{[$F_7$]}}&\cellcolor[gray]{0.9} &   \cellcolor[gray]{0.9}  \\
\hline
 $L=\{1,2\}$& $\frac{7}{2}$& 4 &$\frac{13}{2}$& 7 &$\frac{19}{2}$&11 {\bf{[$F_8$]}}&\cellcolor[gray]{0.9}    \\
\hline
 $L=\{2,3\}$&2 &$\frac{7}{2}$&5  &$\frac{15}{2}$& 8 &$\frac{21}{2}$  &12 {\bf{[$F_9$]}}   \\ \hline
\end{tabular}
\caption{The target value $t(n)-1$ or $t(n)$, where the graphs in squared brackets are only tight examples (up to reversing-equivalent relation)}\label{table1}
\end{table}

\renewcommand{\arraystretch}{1}

If $n \le 4$, then it is easy to check. Suppose that $n=5$ and $e(H)=6$. If $H$ is Hamiltonian, then the chord of a Hamiltonian cycle makes a cycle of length four. Thus $H$ is not a Hamiltonian, and so the longest cycles of $H$ is a triangle. Since $e(H)=6$, $H$ has at least two triangles sharing one vertex. Then $(x,y)$ is a $\{2,3,4\}$-type, which is a contradiction. Thus the column for $n=5$ in the Table~\ref{table1} holds.

Now we assume that $n\ge 6$.
Note that $(x,y)$ is an $L$-type with $|L| = 2$ in $H$. 
If there is no $(x,y)$-path of length $\modfour{\ell}$ for some $\ell\in L$, then it is easy to see that $H$ is bipartite, so $e(H) \le \frac{3n-6}{2}$ by Lemma~\ref{thm:planar}~(iii).  
Therefore we assume that there are both an $(x,y)$-path of length $\modfour{\ell_1}$ and an $(x,y)$-path of length $\modfour{\ell_2}$ where $L=\{\ell_1,\ell_2\}$. To complete the proof, by Table~\ref{table1}, it remains to show the following cases, where $\{a,b\}=\{x,y\}$: 
\begin{itemize}
\item[(A1)] If $n=6$, $L = \{0,3\}$, and $e(H)=7$, then $H$  is reversing-equivalent to $F_6$.
\item[(A2)] If $n=7$, $L=\{0,1\}$, and $e(H)=9$, then $H$ is reversing-equivalent  to $F_7$.
\item[(A3)] If $n= 8$, $L = \{1, 2\}$, and $e(H) =11$, then $H$  is reversing-equivalent to $F_8$. 
\item[(A4)] If $n=7$ and $L = \{2, 3\}$, then $e(H) \le 8$.
\item[(A5)] If $n=9$, $L = \{2, 3\}$, and $e(H) =12$, then $H$ is reversing-equivalent to $F_9$.
\end{itemize}

We use induction on $n$, that is, all the cases  of Table~\ref{table1} hold when $|V(H)|<n$ for some $n\ge 6$. 

Suppose that $H$ has a pendent edge $e$ and $L=\{i,i+1\}$. Then $e$ must be incident with $x$ or $y$, say $e=xz$, then $H-x$ has less vertices and $(y,z)$ satisfies the same condition of the proposition such that $(y,z)$ is an $\{i-1,i\}$-type in $H-x$. 
Then $e(H)\le e(H-x)+1$. 
By the induction hypothesis and Table~\ref{table1}, it cannot be one of the cases (A1), (A2), and (A3).
The case (A4) holds since $e(H)\le e(H-x)+1\le 7+1$. 
The case (A5) holds,  since $12=e(H)=e(H-x)+1=11+1$ and $H-x$ is reversing-equivalent to $F_8$ by (A3), which implies that $H$ is reversing-equivalent to $F_{9}$.

We suppose that $\delta(H)\ge 2$. 
Suppose that there are a cycle $C_x$ containing $x$ and a cycle $C_y$ containing $y$ such that $y\not\in V(C_x)$, $x\not\in V(C_y)$, and 
$|V(C_x)\cap V(C_y)|\le 1$. Then $\ell(C_x)+\ell(C_y)\le n-1$.
If both $\ell(C_x)$ and $\ell(C_y)$ are odd, then there are three $(x,y)$-paths of different lengths modulo $4$ by Table~\ref{table:basic:1201/2303}.  
Since $n\le 9$, one of $C_x$ or $C_y$ is a triangle and the other has even length, and this length must be $6$. Say $C_x$ is a triangle. It must be $n\ge 8$. Then it is not difficult to check that if $n=8$ then $(x,y)$ cannot be a $\{1,2\}$-type in $H$, and if $n=9$ then $(x,y)$ cannot be a $\{2,3\}$-type in $H$.
Thus each case of (A1)$\sim$(A5) is vacuously true.
Therefore, we assume that for a cycle $C_x$ containing $x$ and a cycle $C_y$ containing $y$ such that $|V(C_x)\cap V(C_y)|\le 1$,
either $x\in V(C_y)$ or $y\in V(C_x)$ (and therefore $|V(C_x)\cap V(C_y)|=1$ and the vertex in $V(C_x)\cap V(C_y)$ is $x$ or $y$).
Since $\delta(H)\ge 2$, $x$ and $y$ are on a same cycle $C$. Take such $C$ as a longest one. Then $\ell(C)\ge 5$, since $G$ has no $4$-cycle.   
Let $C:v_1v_2\cdots v_{\ell(C)}v_1$. The following collects some observations.
\begin{itemize}
\item If $\ell(C)=5$, then by the maximality of $\ell(C)$, every vertex not on $C$ has at most one neighbor in $C$.
\item  If $\ell(C)\in\{6,7\}$, then every vertex $z$ not on $C$ has at most two neighbors in $C$ and if $z$ has two neighbors, then they are $v_i$ and $v_{i+3}$ for some $i$, where the subscripts are taken modulo $\ell(C)$.

\end{itemize}  

Suppose the case (A1), that is, $n=6$, $L = \{0,3\}$, and $e(H)=7$.
Then $\ell(C)=6$ and $C$ has only one chord, which implies that $H$ is reversing-equivalent to $F_6$.

Suppose the case (A2), that is, $n=7$, $L=\{0,1\}$, and $e(H)=9$. 
Since $L=\{0,1\}$, $\ell(C)\in\{5,6\}$. 
Suppose that $\ell(C)=5$. Then two vertices not on $C$ form an ear of length three, and so $e(H)\le 8$. 
Suppose that $\ell(C)=6$. Then we may assume that $x=v_1$ and $y=v_2$, since $L=\{0,1\}$.
By Table~\ref{table1}, $H[V(C)]$ has $6$ vertices and at most $7$ edges. The vertex $z$ in $V(H)-V(C)$ has at most two neighbors in $C$. Since $e(H)=9$, $H[V(C)]$ has exactly $7$ edges and $z$ is adjacent to $v_i$ and $v_{i+3}$ for some $i$. Then the only possible case is that
$H$ is reversing-equivalent to $F_7$.
 
Suppose the case (A3), that is, $n = 8$, $L=\{1,2\}$, and $e(H)=11$. 
Then $\ell(C)\in\{6,7\}$, since $(x,y)$ is a $\{1,2\}$-type in $H$. 
Suppose that $\ell(C)=6$. Then we may assume that $x=v_1$ and $y=v_2$.
Thus $C$ has at most one chord and a vertex not on $C$ has at most one neighbor in $C$.
Therefore the vertices not on $C$ form an ear of length three. 
Then $e(H)\le 6+1+3=10$.
Suppose that $\ell(C)=7$. Since $e(H[V(C)])\le 9$ by Table~\ref{table1}, $e(H)= 11$ implies that $e(H[V(C)])= 9$ and the vertex $z$ not on $C$ has exactly two neighbors in $C$. 
By (A2), $H[V(C)]$ is reversing-equivalent to $F_7$. 
Then $H$ is reversing-equivalent to $F_8$. 

To check the cases (A4) and (A5), suppose that $L=\{2,3\}$.
Since there is no $7$-cycle containing both $x$ and $y$,  $\ell(C)\in \{5,6,9\}$. 
Suppose that $\ell(C)=5$. Then we may assume that $x=v_0$ and $y=v_2$. 
Thus $C$ has no chord and each ear of $C$ must have ends $x$ and $y$, which is a contradiction to the maximality of $C$.
Suppose that $\ell(C)=6$. Then we may assume that $x=v_0$ and $y=v_3$.  Thus a vertex not on $C$ has at most two neighbors in $C$ and, if it has two, the neighbors must be $x$ and $y$. Therefore there is at most one vertex with two neighbors in $C$, since $H$ has no $4$-cycle.
If $n=7$, then the neighbors of the vertex not on $C$ must be $x$ and $y$, and so $C$ has no chord and $e(H)\le 6+2=8$, which implies that (A4) holds. If $n=9$, then by the maximality of $C$, an ear of $C$ has length at most three, which implies that $e(H)\le 6+3+2=11$.
Suppose that $\ell(C)=9$. Then $n=9$ and we may assume that $x=v_0$ and $y\in\{v_2,v_3\}$. 
Similar to the previous case, by tedious case checking, we can show that $C$ cannot have three chords, and so $e(H)\le 11$. 
\end{proof}

\begin{thebibliography}{99}

\bibitem{BAI2025}
{\sc Y. Bai, B. Li, Y. Pan and S. Zhang},
\newblock On graphs without cycles of length 1 modulo 3,
\newblock {\it arXiv}: 2503:03504.


\bibitem{BBollobas}
{\sc B. Bollob\'{a}s}.
\newblock Cycles modulo $k$, \newblock{\it Bulletin of the London Mathematical Society}, 
9(1):97--98, 1977.


\bibitem{CS94} 
{\sc G. T. Chen and A. Saito},
\newblock Graphs with a cycle of length divisible by three, 
\newblock{\it Journal of Combinatorial Theory, Series B}, 60(2):277--292, 1994.


\bibitem{DKOT} 
{\sc N. Dean, A. Kaneko, K. Ota, and B. Toft},
\newblock Cycles modulo $3$,
\newblock{\it Dimacs Technical Report}, 91(32), 1991. 

\bibitem{DLS93}
{\sc N. Dean, L. Lesniak, and A. Saito}, 
\newblock Cycles of length $0$ modulo $4$ in graphs,
\newblock {\it Discrete Mathematics}, 121(1--3):37--49, 1993. 

\bibitem{BurrErdos}
{\sc P. Erd\H{o}s}, 
\newblock 
Some recent problems and results in graph theory, combinatorics, and number theory, 
\newblock In {\it Proc. Seventh SE Conf. Combinatorics, Graph Theory and Computing, Utilitas Math}, 3--14, 1976.


\bibitem{2conn2004}
{\sc G. Fan, X. Lv, and P. Wang}, 
\newblock Cycles in $2$-connected graphs,
\newblock{\it Journal of Combinatorial Theory, Series B}, 92(2):379--394, 2004.
 

\bibitem{GLMX2024}
{\sc J. Gao, B. Li, J. Ma, and T. Xie}, 
\newblock 
On two cycles of consecutive even lengths, 
\newblock{\it Journal of Graph Theory}, 106(2):225--238, 2024.

\bibitem{0mod4cycle2025}
{\sc E. Gy\H{o}ri, B. Li, N. Salia, C. Tompkins, K. Varga, and M. Zhu},
\newblock On graphs without cycles of length $0$ modulo $4$,
\newblock {\it arXiv}:2312.09999.

\bibitem{LPS2025}
{\sc B. Li, Y. Pan and L. Shi},
\newblock A note on two cycles of consecutive even lengths in graphs,
\newblock{\it arXiv}:2506.08692


\bibitem{OBW13}
{\sc O. R. Oellermann, L. W. Beineke, and R. J. Wilson},
\newblock Menger's theorem,
\newblock In {\it Topics in Structural Graph Theory}, Cambridge University Press, 13--39, 2013.

\bibitem{Saito92}
{\sc A. Saito}, 
\newblock Cycles of length $2$ modulo $3$ in graphs, 
\newblock{\it Discrete Mathematics}, 
101(1--3):285--289, 1992.

\bibitem{Sudakov2017} 
{\sc B. Sudakov and J. Verstra\"{e}te},
\newblock The extremal function for cycles of length $\ell$ mod $k$,
\newblock {\it The Electronic Journal
of Combinatorics}, 24(1), 2017.


\bibitem{V2016}
{\sc J. Verstra\"{e}te}, 
\newblock Extremal problems for cycles in graphs, 
\newblock In \textit{Recent trends in combinatorics}, Springer, 83--116, 2016.


\end{thebibliography}
\end{document}